\documentclass[a4paper,10pt]{article}

\usepackage{amssymb}
\usepackage[centertags]{amsmath}
\usepackage{amsfonts}
\usepackage{amsthm}
\usepackage{MnSymbol} 
\usepackage{url}
\newtheorem{theorem}{Theorem}[section]
\newtheorem{corollary}[theorem]{Corollary}
\newtheorem{lemma}[theorem]{Lemma}
\newtheorem{proposition}[theorem]{Proposition}
\newtheorem{definition}[theorem]{Definition}
\newtheorem{example}[theorem]{Example}
\newtheorem{remark}[theorem]{Remark}
\newtheorem{conjecture}[theorem]{Conjecture}

\usepackage{array,multirow,graphicx}
\usepackage{float}
\usepackage{tikz}
\usepackage{tikz-3dplot}
\usetikzlibrary{arrows.meta}
\usetikzlibrary{shapes,arrows}
\usepackage{wrapfig}
\tikzstyle{block}=[draw opacity=0.7,linewidth=1.4cm]{\usetikzlibrary{arrows,shapes}}
\usepackage{xcolor}
\usepackage{color, colortbl}
\usepackage{lscape}
\usepackage{multirow}
\definecolor{classicrose}{rgb}{0.98, 0.8, 0.91}
\definecolor{cottoncandy}{rgb}{1.0, 0.74, 0.85}
\definecolor{mediumchampagne}{rgb}{0.95, 0.9, 0.67}
\definecolor{maize}{rgb}{0.98, 0.93, 0.37}
\definecolor{celadon}{rgb}{0.67, 0.88, 0.69}
\definecolor{darkseagreen}{rgb}{0.56, 0.74, 0.56}
\definecolor{pastelyellow}{rgb}{0.99, 0.99, 0.59}
\definecolor{sandstorm}{rgb}{0.93, 0.84, 0.25}
\definecolor{skyblue}{rgb}{0.53, 0.81, 0.92}
\definecolor{royalblue}{rgb}{0.25, 0.41, 0.88}
\definecolor{celadon}{rgb}{0.67, 0.88, 0.69}
\definecolor{darkcyan}{rgb}{0.0, 0.55, 0.55}
\usepackage{multirow}
\usepackage{subcaption} 
\usepackage{chessboard} 

\usepackage{authblk} 
\usepackage{array}
\usepackage{array}
\newcolumntype{L}[1]{>{\raggedright\let\newline\\\arraybackslash\hspace{0pt}}m{#1}}
\newcolumntype{C}[1]{>{\centering\let\newline\\\arraybackslash\hspace{0pt}}m{#1}}
\newcolumntype{R}[1]{>{\raggedleft\let\newline\\\arraybackslash\hspace{0pt}}m{#1}}

\title{A family of regular integral graphs and its application to the $n$-Queens' graph}

\author[1,2]{Domingos M. Cardoso}
\author[1,2]{In\^es Ser\^odio Costa}
\author[1,2]{Rui Duarte}

\affil[1]{\small Centro de Investiga\c{c}\~{a}o e Desenvolvimento em Matem\'atica e Aplica\c{c}\~{o}es}
\affil[2]{\small Departamento de Matem\'atica, Universidade de  Aveiro, 3810-193, Aveiro, Portugal.}

\begin{document}
\maketitle

\begin{abstract}
A family of regular integral graphs introduced in [I.F.S. Costa, The $n$-Queens graph and its generalizations, Ph.D. Thesis, University of Aveiro 2024], denoted by ${\cal T}(n)$ and herein called triangular graphs, is analysed. In this analysis, the consistent structure of the graph spectra and the patterns of the corresponding eigenvectors are highlighted. The properties of these graphs are examined and applied to the decomposition of the $n$-Queens' graph into three distinct families: a family of a single graph whose components are two triangular graphs, ${\cal T}(n)$ and ${\cal T}(n-1)$, a family of a single graph whose components are cliques and a family of complete bipartite graphs. Finally, using Weyl's inequalities, we introduce some techniques to establish lower and upper bounds on the eigenvalues of the $n$-Queens' graph.
\end{abstract}

\medskip


\medskip


\medskip

\section{Introduction}
An integral graph is a graph whose spectrum is entirely integral. The study of integral graphs was introduced by F. Harary and A. J. Schwenk in \cite{1974Harary}, where they ask “which graphs have integral spectra?”, with the remark that the problem for general graphs appears intractable. As it is referred in \cite{2002BCRSS}, the number of these graphs is not only infinite, but one can find them in all classes of graphs and among graphs of all orders. However, despite the existence of trivial examples of integral graphs, as it is the case of complete graphs (see further examples in \cite{1974Harary}), they are very rare and difficult to be found \cite{2002BCRSS}. Observe that, according to \cite[Cor. 1]{1974Harary}, if a regular graph is integral, then its complement is also integral. So, the graph complement of a triangular graph ${\cal T}(n)$ (which as we will see is integral) is also integral. In \cite{2002BCRSS} and \cite{2005Wang}, a survey of results on integral graphs is presented.\\

The designation triangular graph in some publications is used for graphs with distinct properties, as it is the case of the one introduced by Hoffman in \cite{Hoffman1960}, in the context of the study of the uniqueness of the triangular association scheme. These graphs are also referred in the book of Brualdi and Ryser \cite[p. 152]{BrualdiRyser1991}, where $T(m)$ denotes a triangular graph which is defined as being the line graph of the complete graph $K_m$, ($m \ge 4$). Thus the vertices of $T(m)$ may be identified as the $2$-subsets of $\{1, 2, \dots, m\}$ and two vertices are adjacent in $T(m)$ if and only if the corresponding $2$-subsets have a nonempty intersection. As it is referred in  \cite{BrualdiRyser1991}, these graphs are strongly regular graphs on the parameters $(n,d,\lambda,\mu)=(\frac{m(m-1)}{2}, 2(m-2), m-2, 4)$. Therefore, since ${\cal T}(n)$ is not strongly regular, this graph is not isomorphic to $T(n)$. However, both graphs have some close parameters, as it can be seen in Table~\ref{parameters}.
\begin{table}[h] 
\begin{center}
\begin{tabular}{|l||c|c|}
\hline         & Triangular graph ${\cal T}(m)$ & Triangular graph $T(m), m \ge 4$ \\ \hline \hline
      order    & $n = \frac{m(m+1)}{2}$         & $n = \frac{m(m-1)}{2}$         \\ \hline
      degree   & $d = 2(m-1)$                   & $d = 2(m-2)$                   \\ \hline
      diameter & $D=2$, for $m \ge 3$           & $D=2$                          \\ \hline
\end{tabular}
\caption{Some parameters of the graphs ${\cal T}(m)$ and $T(m)$.}\label{parameters}
\end{center}
\end{table}

In this paper the results of Section~\ref{sec-SpectralPropTn} follow part of the ones introduced in \cite[Chapter 6]{Costa2024}. Even so, for the reader's convenience, some proofs are herein recalled. The properties of triangular graphs of order $\frac{n(n+1)}{2}$, ${\cal T}(n)$, are explored and applied to the decomposition of the $n$-Queens' graph into graphs belonging to three families -- a family of a single graph whose components are two triangular graphs, ${\cal T}(n)$ and ${\cal T}(n-1)$, a family of a single graph whose components are cliques and a family of complete bipartite graphs. Finally, using the Weyl's inequalities, some techniques to determine lower and upper bounds on the eigenvalues of the $n$-Queens' graph are introduced. The notation used throughout the text, with the exception of very specific cases, is the one that is followed in \cite{BrouwerHaemers12} and \cite{2009CRS}.

\section{Triangular board and triangular graph $\mathcal{T}(n)$}\label{Sec-Introducti}
Let $\mathbb{T}_n$ be the board of $\frac{n(n+1)}{2}$ points in a triangular shape with $n$ dots on each side. For $n=1,2,3,4$, such arrangement is represented in Figure \ref{fig-numerotriangular}. The rows of $\mathbb{T}_n$ are labeled from $1$ to $n$, from the top to the bottom and, in the $i^{\text{th}}$ row, its dots are labeled from $1$ to $i$, from the left to the right. Then, the $j^{\text{th}}$ entry (dot) of the $i^{\text{th}}$ row of $\mathbb{T}(n)$ is identified by the ordered pair $(i,j)$.
	
	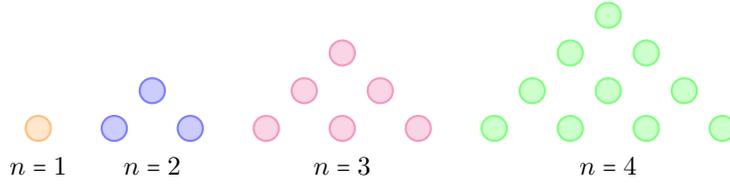
\begin{figure}[h!]
		\centering
		\begin{tikzpicture}[thick]
			\node at (1,1) [circle,draw=orange!50,fill=orange!20] {};
			\draw (1,0.5) node {$n=1$};
			
			\node at (2,1) [circle,draw=blue!50,fill=blue!20] {};
			\node at (3,1) [circle,draw=blue!50,fill=blue!20] {};
			\node at (2.5,1.5) [circle,draw=blue!50,fill=blue!20] {};
			\draw (2.5,0.5) node {$n=2$};
			
			\node at (4,1) [circle,draw=magenta!50,fill=magenta!20] {};
			\node at (5,1) [circle,draw=magenta!50,fill=magenta!20] {};
			\node at (6,1) [circle,draw=magenta!50,fill=magenta!20] {};
			\node at (4.5,1.5) [circle,draw=magenta!50,fill=magenta!20] {};
			\node at (5.5,1.5) [circle,draw=magenta!50,fill=magenta!20] {};
			\node at (5,2) [circle,draw=magenta!50,fill=magenta!20] {};
			\draw (5,0.5) node {$n=3$};
			
			\node at (7,1) [circle,draw=green!50,fill=green!20] {};
			\node at (8,1) [circle,draw=green!50,fill=green!20] {};
			\node at (9,1) [circle,draw=green!50,fill=green!20] {};
			\node at (10,1) [circle,draw=green!50,fill=green!20] {};
			\node at (7.5,1.5) [circle,draw=green!50,fill=green!20] {};
			\node at (8.5,1.5) [circle,draw=green!50,fill=green!20] {};
			\node at (9.5,1.5) [circle,draw=green!50,fill=green!20] {};
			\node at (8,2) [circle,draw=green!50,fill=green!20] {};
			\node at (9,2) [circle,draw=green!50,fill=green!20] {};
			\node at (8.5,2.5) [circle,draw=green!50,fill=green!20] {};
			\draw (8.5,0.5) node {$n=4$};
		\end{tikzpicture}
		\caption{Representation of $\mathbb{T}_n$ for $n=1,2,3,4$.} \label{fig-numerotriangular}
	\end{figure}
	
	Note that, for $n\in\mathbb{N}$, the number of dots of $\mathbb{T}_n$ is the \emph{$n^{\text{th}}$ triangular number} $T(n)=\frac{n(n+1)}{2}$.
	
	\medskip
	
	Let us define the $n^{\text{th}}$ triangular graph as $\mathcal{T}(n)=(V,E)$ with
\begin{align*}
		V & = \left\{ (i,j) \mid i \in [n] \land j \in [i] \right\} \\
        \intertext{and}
		E & = \left\{ ((i,j),(k,\ell)) \in V^2 \mid (i=k \vee j = \ell \vee i-j=k-\ell) \land (i,j) \neq (k,\ell) \right\}.
	\end{align*}
	
	The vertex $(i,j)\in V$ can also be denoted by the label $\ell=T(i-1)+j\in[T(n)]$. These two different labelings are presented in Figure \ref{fig-exampleTn} for $\mathcal{T}(4)$. The vertices of $\mathcal{T}(n)$ will be represented in an identical way as the dots in Figure \ref{fig-numerotriangular}: the vertex $(i,j)$ is in the entry $(i,j)$ of $\mathbb{T}_n$. The row, column and diagonal of the vertex $(i,j)$ are called $i^\text{th}$ row, $j^\text{th}$ column and $(i-j)^\text{th}$ diagonal. So, two distinct vertices are neighbors if they are in the same row, column or diagonal.
	
	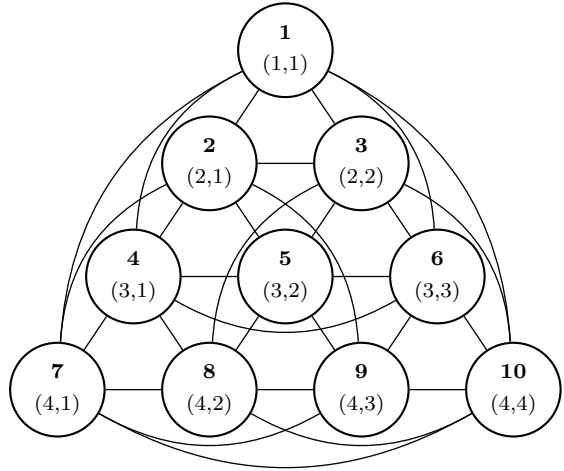
\begin{figure}[h!]
		\centering
		\begin{tikzpicture}
			\begin{scope}[every node/.style={circle,thick,draw,minimum size=0.8cm},every text node part/.style={align=center}]
				\node (1) at (0,0) {\footnotesize\textbf{1}\\ \footnotesize(1,1)};
				\node (2) at (-1,-1.5) {\footnotesize\textbf{2}\\ \footnotesize(2,1)};
				\node (3) at (1,-1.5) {\footnotesize\textbf{3}\\\footnotesize(2,2)};
				\node (4) at (-2,-3) {\footnotesize\textbf{4}\\ \footnotesize(3,1)};
				\node (5) at (0,-3) {\footnotesize\textbf{5}\\ \footnotesize(3,2)};
				\node (6) at (2,-3) {\footnotesize\textbf{6}\\ \footnotesize(3,3)};
				\node (7) at (-3,-4.5) {\footnotesize\textbf{7}\\ \footnotesize(4,1)};
				\node (8) at (-1,-4.5) {\footnotesize\textbf{8}\\ \footnotesize(4,2)};
				\node (9) at (1,-4.5) {\footnotesize\textbf{9}\\ \footnotesize(4,3)};
				\node (10) at (3,-4.5) {\footnotesize\textbf{10}\\ \footnotesize(4,4)};
			\end{scope}
			
			\begin{scope}[>={Stealth[black]},
				every edge/.style={draw=black,line width=0.5pt}]
				\path (1) edge node {} (2);
				\path (1) edge node {} (3);
				\path (1) edge[bend right=30] node {} (4);
				\path (1) edge[bend left=30] node {} (6);
				\path (1) edge[bend right=30] node {} (7);
				\path (1) edge[bend left=30] node {} (10);
				
				\path (2) edge node {} (3);
				\path (2) edge node {} (4);
				\path (2) edge node {} (5);
				\path (2) edge[bend right=30] node {} (7);
				\path (2) edge[bend left=30] node {} (9);
				
				\path (3) edge node {} (5);
				\path (3) edge node {} (6);
				\path (3) edge[bend right=30] node {} (8);
				\path (3) edge[bend left=30] node {} (10);
				
				\path (4) edge node {} (5);
				\path (4) edge node {} (7);
				\path (4) edge[bend right=30] node {} (6);
				\path (4) edge node {} (8);
				
				\path (5) edge node {} (6);
				\path (5) edge node {} (8);
				\path (5) edge node {} (9);
				
				\path (6) edge node {} (9);
				\path (6) edge node {} (10);
				
				\path (7) edge node {} (8);
				\path (7) edge[bend right=30] node {} (9);
				\path (7) edge[bend right=30] node {} (10);
				
				\path (8) edge node {} (9);
				\path (8) edge[bend right=30] node {} (10);
				
				\path (9) edge node {} (10);
				
			\end{scope}
		\end{tikzpicture}
		\caption{$\mathcal{T}(4)$ and its two vertex labelings.}\label{fig-exampleTn}
	\end{figure}
	
	\medskip

	By definition, the order of $\mathcal{T}(n)$, which is the number of vertices, is $T(n)=\frac{n(n+1)}{2}$.
	
	\begin{proposition}\label{prop-2n-2-regulargraph}
		$\mathcal{T}(n)$ is a $(2n-2)$-regular graph.
	\end{proposition}
	\begin{proof}
Since the sets of vertices in the $i^\text{th}$ row, $j^\text{th}$ column and $(i-j)^\text{th}$ diagonal are
\[
\{ (i,1), \ldots, (i,i) \}, \ \ \{ (j,j), \ldots, (n,j) \}, \ \ \text{and} \ \ \{ (1+i-j,1), \ldots, (n,n-(i-j)) \},
\]
respectively,
\[
d((i,j)) = (i-1) + (n-j) + (n-i+j-1) = 2n-2. \qedhere
\]
	\end{proof}
	
    By Proposition \ref{prop-2n-2-regulargraph}, the size of $\mathcal{T}(n)$, which is the number of edges, is $\frac{(2n-2)|V(\mathcal{T}(n))|}{2}=\frac{n^3-n}{2}$.

In the next section, it will be proved that the triangular graphs are integral.

Before that, it is worth recalling the following notions which appear in \cite{CardosoCostaDuarte2023}.

An edge clique partition (ECP for short) of a graph $G$ is a partition of the set of edges such that each part induces a complete graph. Considering a graph $G$ and an ECP $P=\{ E_i \mid i \in I \}$, it follows that $V_i=V(G[E_i])$ is a clique of $G$ for every $i \in I$. For any $v \in V(G)$, the clique degree of $v$ relative to $P$, denoted $m_v(P)$, is the number of cliques $V_i$ containing the vertex $v$, and the maximum clique degree of $G$ relative to $P$, denoted $m_G(P)$, is the maximum of clique degrees of the vertices of $G$ relative to $P$.

It is also worth to recall the next theorem which also appears in \cite{CardosoCostaDuarte2023}.

\begin{theorem} {\rm \cite{CardosoCostaDuarte2023}} \label{Main_result_1}
	Let $P=\{ E_i \mid i \in I \}$ be an ECP of a graph $G$, $m=m_G(P)$ and $m_v=m_v(P)$ for every $v \in V(G)$. Then
	\begin{enumerate}
		\item If $\mu$ is an eigenvalue of $G$, then $\mu \ge -m$. \label{marca}
		\item $-m$ is an eigenvalue of $G$ if and only if there exists a vector $x \ne \mathbf{0}$ such that
		\begin{enumerate}
			\item $\sum \limits_{j \in V(G[E_i])} x_j = 0$, for every $i \in I$ and \label{cond1}
			\item $\forall v \in V(G) \;\; x_v=0$ whenever $m_v \ne m$. \label{cond2}
		\end{enumerate}\label{marca2}
		In the positive case, $x$ is an eigenvector associated with $-m$.
	\end{enumerate}
\end{theorem}

\section{Spectral properties of $\mathcal{T}(n)$}\label{sec-SpectralPropTn}

	Since the vectors in this section have a triangular number of components, they are written over $\mathbb{T}_n$, as in the previous section. Therefore, the $a^{\text{th}}$ coordinate of a vector $v$ will be displayed at the entry $(i,j)$ of $\mathbb{T}_n$ for $a=T(i-1)+j$ and it will be denoted by $v_a$ or $v_{(i,j)}$.
	
	\begin{example}
    The vector $v=[1,2,3,4,5,6,7,8,9,10]^T$ consists of $T(4)=10$ components. This type of vector will be displayed on a triangular board, as illustrated in Figure \ref{fig-example-vectorTn}.
		\begin{figure}[h]
			\centering
			\begin{tikzpicture}[scale=0.85]
				\begin{scope}[every node/.style={circle,thick,draw,minimum size=0.8cm},every text node part/.style={align=center}]
					\node (1) at (4.5,0) [circle] {1};
					\node (2) at (4,-0.9) [circle] {2};
					\node (3) at (5,-0.9) [circle] {3};
					\node (4) at (3.5,-1.8) [circle] {4};
					\node (5) at (4.5,-1.8) [circle] {5};
					\node (6) at (5.5,-1.8) [circle] {6};
					\node (7) at (3,-2.7) [circle] {7};
					\node (8) at (4,-2.7) [circle] {8};
					\node (9) at (5,-2.7) [circle] {9};
					\node (10) at (6,-2.7) [circle] {10};
				\end{scope}
				\draw (2.2,-1.35) node {$v=$};
			\end{tikzpicture}
			\caption{A vector on the triangular board.}\label{fig-example-vectorTn}
		\end{figure}
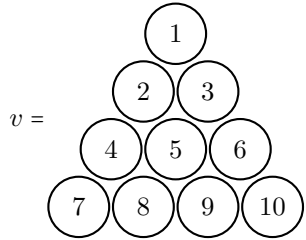
		So, for example, the $5^\text{th}$ component of $v$, which is $5$, will be denoted $v_{(3,2)}$ because it is in the $3^\text{th}$ row and $2^\text{nd}$ column.
	\end{example}
	
	\medskip
	
	Consider the following definition.
	
	\begin{definition}\label{Def-v+v-}
		Let $v\in\mathbb{R}^{T(n)}$ be a vector for $n\in\mathbb{N}$. The positive rotation and the negative rotation of $v$ are denoted by $v^+$ and $v^-$, respectively, and are defined as follows
		\[
        v^+_{(i,j)}=v_{(n-i+j,n-i+1)} \quad \text{and} \quad v^-_{(i,j)}=v_{(n-i+1,i-j+1)}.
        \]
	\end{definition}
	
	\begin{remark}
		Note that, considering the representation of the vectors in $\mathbb{T}_n$, $v^+$ and $v^-$ correspond to a $120^\circ$ and a $-120^\circ$ $(240^\circ)$ rotation of $v$, respectively. So, these transformations fulfill the same properties as rotations of order $3$. For example $(v^+)^-=v$, $(v^+)^+=v^-$, $((v^+)^+)^+=v$. In Figure \ref{fig-vector-rotation}, an example of these vector rotations is presented.
	\end{remark}
	
	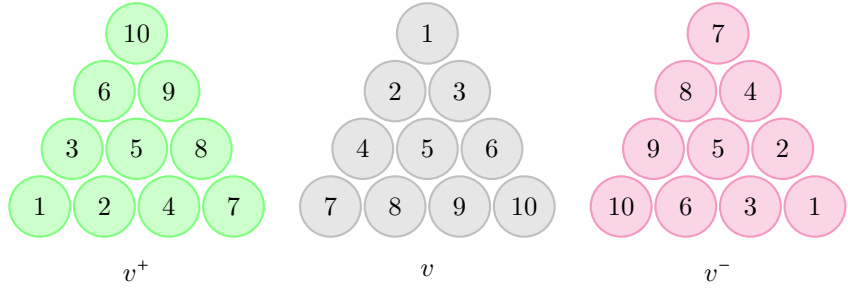
\begin{figure}[h!]
		\centering
		\begin{tikzpicture}[scale=0.85]
			\begin{scope}[every node/.style={circle,thick,draw,minimum size=0.8cm},every text node part/.style={align=center}]
				\node (1+) at (0,0) [circle,draw=green!50,fill=green!20] {10};
				\node (2+) at (-0.5,-0.9) [circle,draw=green!50,fill=green!20] {6};
				\node (3+) at (0.5,-0.9) [circle,draw=green!50,fill=green!20] {9};
				\node (4+) at (-1,-1.8) [circle,draw=green!50,fill=green!20] {3};
				\node (5+) at (0,-1.8) [circle,draw=green!50,fill=green!20] {5};
				\node (6+) at (1,-1.8) [circle,draw=green!50,fill=green!20] {8};
				\node (7+) at (-1.5,-2.7) [circle,draw=green!50,fill=green!20] {1};
				\node (8+) at (-0.5,-2.7) [circle,draw=green!50,fill=green!20] {2};
				\node (9+) at (0.5,-2.7) [circle,draw=green!50,fill=green!20] {4};
				\node (10+) at (1.5,-2.7) [circle,draw=green!50,fill=green!20] {7};
				
				\node (1) at (4.5,0) [circle,draw=gray!50,fill=gray!20] {1};
				\node (2) at (4,-0.9) [circle,draw=gray!50,fill=gray!20] {2};
				\node (3) at (5,-0.9) [circle,draw=gray!50,fill=gray!20] {3};
				\node (4) at (3.5,-1.8) [circle,draw=gray!50,fill=gray!20] {4};
				\node (5) at (4.5,-1.8) [circle,draw=gray!50,fill=gray!20] {5};
				\node (6) at (5.5,-1.8) [circle,draw=gray!50,fill=gray!20] {6};
				\node (7) at (3,-2.7) [circle,draw=gray!50,fill=gray!20] {7};
				\node (8) at (4,-2.7) [circle,draw=gray!50,fill=gray!20] {8};
				\node (9) at (5,-2.7) [circle,draw=gray!50,fill=gray!20] {9};
				\node (10) at (6,-2.7) [circle,draw=gray!50,fill=gray!20] {10};
				
				\node (1-) at (9,0) [circle,draw=magenta!50,fill=magenta!20] {7};
				\node (2-) at (8.5,-0.9) [circle,draw=magenta!50,fill=magenta!20] {8};
				\node (3-) at (9.5,-0.9) [circle,draw=magenta!50,fill=magenta!20] {4};
				\node (4-) at (8,-1.8) [circle,draw=magenta!50,fill=magenta!20] {9};
				\node (5-) at (9,-1.8) [circle,draw=magenta!50,fill=magenta!20] {5};
				\node (6-) at (10,-1.8) [circle,draw=magenta!50,fill=magenta!20] {2};
				\node (7-) at (7.5,-2.7) [circle,draw=magenta!50,fill=magenta!20] {10};
				\node (8-) at (8.5,-2.7) [circle,draw=magenta!50,fill=magenta!20] {6};
				\node (9-) at (9.5,-2.7) [circle,draw=magenta!50,fill=magenta!20] {3};
				\node (10-) at (10.5,-2.7) [circle,draw=magenta!50,fill=magenta!20] {1};
			\end{scope}
			\draw (0,-3.7) node {$v^+$};
			\draw (4.5,-3.7) node {$v$};
			\draw (9,-3.7) node {$v^-$};
		\end{tikzpicture}
		\caption{A vector $v$ and its positive (in green) and negative (in magenta) rotations.}\label{fig-vector-rotation}
	\end{figure}

	Consider the following proposition which states a relation between the eigenvalues and eigenvectors of a graph $G$, and an automorphism of $G$.
	
	\begin{proposition} {\rm \cite[Corollary 2.9]{LauriScapellato2003}} \label{Prop-EigAutomGraph}
		Let $x$ be an eigenvector of $A$ associated with
        $\lambda$ and $P$ a permutation matrix corresponding to an automorphism of $G$. Then $Px$ is also an eigenvector of $A$ associated with $\lambda$.
	\end{proposition}
	
	\begin{corollary}
		If $v\in\mathbb{R}^{T(n)}$, for $n\in\mathbb{N}$, is an eigenvector of $\mathcal{T}(n)$, then their two rotations, $v^+$ and $v^-$, are also eigenvectors associated with the same eigenvalue.
	\end{corollary}
	\begin{proof}
    Considering Definition~\ref{Def-v+v-} and Proposition~\ref{Prop-EigAutomGraph}, the result follows.
	\end{proof}
	
	
	Consider the vectors $R_{n,r}, C_{n,c},D_{n,d}\in\mathbb{R}^{T(n)}$, with $n\in\mathbb{N}$, $r,c\in[n]$ and $d\in[0,n-1]$, defined as follows.
    \begin{equation}
		R_{n,r}(i,j) = \delta_{ir}, \quad
		C_{n,c}(i,j) = \delta_{jc}, \quad \text{and} \quad
		D_{n,d}(i,j) = \delta_{(i-j)d}
	\end{equation}
    where $\delta$ is the Kronecker delta.

	We call these vectors \emph{$RCD$-vectors}. See Figure~\ref{fig-RCD-vectors}.
	
	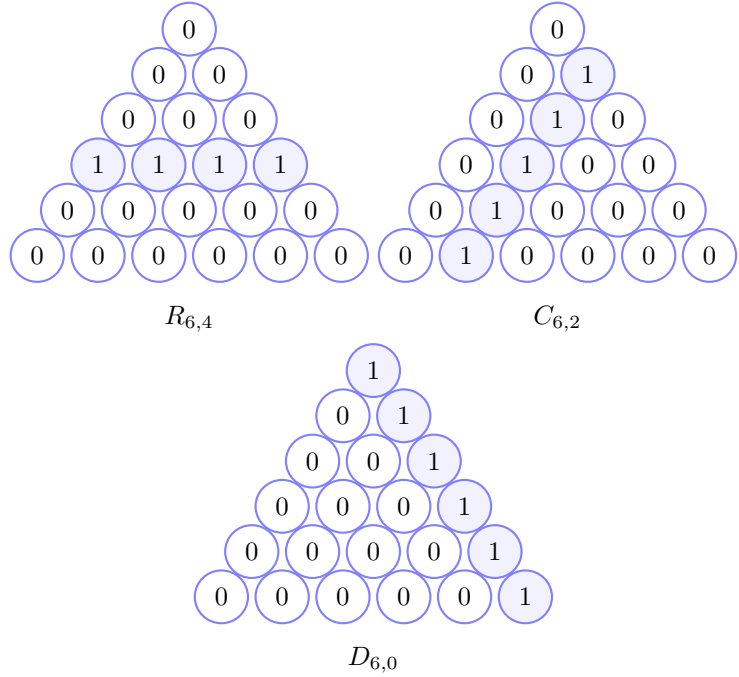
\begin{figure}[h!]
		\centering
		\begin{tikzpicture}[scale=0.8]
			\begin{scope}[every node/.style={circle,thick,draw,minimum size=0.7cm},every text node part/.style={align=center}]
				\node (1) at (0,0) [circle,draw=blue!50] {0};
				\node (2) at (-0.5,-0.75) [circle,draw=blue!50] {0};
				\node (3) at (0.5,-0.75) [circle,draw=blue!50] {0};
				\node (4) at (-1,-1.5) [circle,draw=blue!50] {0};
				\node (5) at (0,-1.5) [circle,draw=blue!50] {0};
				\node (6) at (1,-1.5) [circle,draw=blue!50] {0};
				\node (7) at (-1.5,-2.25) [circle,draw=blue!50,fill=blue!5] {1};
				\node (8) at (-0.5,-2.25) [circle,draw=blue!50,fill=blue!5] {1};
				\node (9) at (0.5,-2.25) [circle,draw=blue!50,fill=blue!5] {1};
				\node (10) at (1.5,-2.25) [circle,draw=blue!50,fill=blue!5] {1};
				\node (11) at (-2,-3) [circle,draw=blue!50] {0};
				\node (12) at (-1,-3) [circle,draw=blue!50] {0};
				\node (13) at (0,-3) [circle,draw=blue!50] {0};
				\node (14) at (1,-3) [circle,draw=blue!50] {0};
				\node (15) at (2,-3) [circle,draw=blue!50] {0};
				\node (16) at (-2.5,-3.75) [circle,draw=blue!50] {0};
				\node (17) at (-1.5,-3.75) [circle,draw=blue!50] {0};
				\node (18) at (-0.5,-3.75) [circle,draw=blue!50] {0};
				\node (19) at (0.5,-3.75) [circle,draw=blue!50] {0};
				\node (20) at (1.5,-3.75) [circle,draw=blue!50] {0};
				\node (21) at (2.5,-3.75) [circle,draw=blue!50] {0};
			\end{scope}
			\draw (0,-4.8) node {$R_{6,4}$};
		\end{tikzpicture}	
		\begin{tikzpicture}[scale=0.8]
			\begin{scope}[every node/.style={circle,thick,draw,minimum size=0.7cm},every text node part/.style={align=center}]
				\node (1) at (0,0) [circle,draw=blue!50] {0};
				\node (2) at (-0.5,-0.75) [circle,draw=blue!50] {0};
				\node (3) at (0.5,-0.75) [circle,draw=blue!50,fill=blue!5] {1};
				\node (4) at (-1,-1.5) [circle,draw=blue!50] {0};
				\node (5) at (0,-1.5) [circle,draw=blue!50,fill=blue!5] {1};
				\node (6) at (1,-1.5) [circle,draw=blue!50] {0};
				\node (7) at (-1.5,-2.25) [circle,draw=blue!50] {0};
				\node (8) at (-0.5,-2.25) [circle,draw=blue!50,fill=blue!5] {1};
				\node (9) at (0.5,-2.25) [circle,draw=blue!50] {0};
				\node (10) at (1.5,-2.25) [circle,draw=blue!50] {0};
				\node (11) at (-2,-3) [circle,draw=blue!50] {0};
				\node (12) at (-1,-3) [circle,draw=blue!50,fill=blue!5] {1};
				\node (13) at (0,-3) [circle,draw=blue!50] {0};
				\node (14) at (1,-3) [circle,draw=blue!50] {0};
				\node (15) at (2,-3) [circle,draw=blue!50] {0};
				\node (16) at (-2.5,-3.75) [circle,draw=blue!50] {0};
				\node (17) at (-1.5,-3.75) [circle,draw=blue!50,fill=blue!5] {1};
				\node (18) at (-0.5,-3.75) [circle,draw=blue!50] {0};
				\node (19) at (0.5,-3.75) [circle,draw=blue!50] {0};
				\node (20) at (1.5,-3.75) [circle,draw=blue!50] {0};
				\node (21) at (2.5,-3.75) [circle,draw=blue!50] {0};
			\end{scope}
			\draw (0,-4.8) node {$C_{6,2}$};
		\end{tikzpicture}
		\begin{tikzpicture}[scale=0.8]
			\begin{scope}[every node/.style={circle,thick,draw,minimum size=0.7cm},every text node part/.style={align=center}]
				\node (1) at (0,0) [circle,draw=blue!50,fill=blue!5] {1};
				\node (2) at (-0.5,-0.75) [circle,draw=blue!50] {0};
				\node (3) at (0.5,-0.75) [circle,draw=blue!50,fill=blue!5] {1};
				\node (4) at (-1,-1.5) [circle,draw=blue!50] {0};
				\node (5) at (0,-1.5) [circle,draw=blue!50] {0};
				\node (6) at (1,-1.5) [circle,draw=blue!50,fill=blue!5] {1};
				\node (7) at (-1.5,-2.25) [circle,draw=blue!50] {0};
				\node (8) at (-0.5,-2.25) [circle,draw=blue!50] {0};
				\node (9) at (0.5,-2.25) [circle,draw=blue!50] {0};
				\node (10) at (1.5,-2.25) [circle,draw=blue!50,fill=blue!5] {1};
				\node (11) at (-2,-3) [circle,draw=blue!50] {0};
				\node (12) at (-1,-3) [circle,draw=blue!50] {0};
				\node (13) at (0,-3) [circle,draw=blue!50] {0};
				\node (14) at (1,-3) [circle,draw=blue!50] {0};
				\node (15) at (2,-3) [circle,draw=blue!50,fill=blue!5] {1};
				\node (16) at (-2.5,-3.75) [circle,draw=blue!50] {0};
				\node (17) at (-1.5,-3.75) [circle,draw=blue!50] {0};
				\node (18) at (-0.5,-3.75) [circle,draw=blue!50] {0};
				\node (19) at (0.5,-3.75) [circle,draw=blue!50] {0};
				\node (20) at (1.5,-3.75) [circle,draw=blue!50] {0};
				\node (21) at (2.5,-3.75) [circle,draw=blue!50,fill=blue!5] {1};
			\end{scope}
			\draw (0,-4.8) node {$D_{6,0}$};
		\end{tikzpicture}
		
		\caption{The vectors $R_{6,4}, C_{6,2}$ and $D_{6,0}$.}\label{fig-RCD-vectors}
	\end{figure}
	
	
	For a vector $v\in \mathbb{R}^{T(n)}$, $n\in\mathbb{N}$, the vectors $S^r_v$, $S^c_v$, $S^d_v \in \mathbb{R}^n$ represent the sum of the components of $v$ by rows, columns, and diagonals, respectively, and could be defined as follows.	
	\begin{equation}
		\begin{split}
			S^r_v (\ell_1) &= \sum_{j=1}^{\ell_1}v(\ell_1,j),\quad \ell_1\in[n]\\
			S^c_v (\ell_2) &= \sum_{i=\ell_2}^{n}v(i,\ell_2),\quad \ell_2\in[n]\\
			S^d_v (\ell_3+1) &= \sum_{j=1}^{n-\ell_3} v(j+\ell_3,j),\quad \ell_3\in[0,n-1]\\
		\end{split}
	\end{equation}
	These vectors will be simply called \emph{sum-vector by row/column/diagonal} of $v$.
	
	\subsection{The largest and the least eigenvalue}\label{sec-leastlargesteigenvalueTn}
	
	In this section, the largest and the least eigenvalues of $\mathcal{T}(n)$ and associated eigenvectors are investigated.
	
	\begin{theorem}\label{Th_Largest}
		For $n\in\mathbb{N}$, the largest eigenvalue of $\mathcal{T}(n)$ is $2n-2$. Furthermore, $2n-2$ is simple and $\mathbf{1}^{T(n)}$ is an eigenvector associated with $2n-2$.
	\end{theorem}
	\begin{proof}
		Since $\mathcal{T}(n)$ is $(2n-2)-$regular, as it is well known (see for instance Theorem 1.1.2 in \cite{2009CRS}) the result follows.
	\end{proof}
	
	\medskip
	
	Let's introduce the family of vectors
	$$
	\mathbf{T}_n = \{t_n^{(x,y)} \in \mathbb{R}^{T(n)} \mid x\in [n-3], y\in[x] \}
	$$
	where $t_n^{(x,y)}$ is the vector defined by
	
	\begin{equation}\label{eq-vectors--3-Tn}
		\big[t_n^{(x,y)}\big]_{(i,j)} = \begin{cases}
			t_{(i-x+1,j-y+1)}, & \text{ if } x\le i\le x+3 \text{ and } y\le j\le y+3, \\
			0, & \text{otherwise;}
		\end{cases}
	\end{equation}
	with $t=t_4^{(1,1)}=[0,1,-1,-1,0,1,0,1,-1,0]$ as presented in Figure~\ref{fig-Tn--3vectors}.
	
	\begin{figure}[h!]
		\centering
		\begin{tikzpicture}[scale=0.8]
			\begin{scope}[every node/.style={circle,thick,draw,minimum size=0.7cm},every text node part/.style={align=center}]
				\node (1) at (0,0) {0};
				\node (2) at (-0.5,-0.75) {1};
				\node (3) at (0.5,-0.75) {-1};
				\node (4) at (-1,-1.5) {-1};
				\node (5) at (0,-1.5) {0};
				\node (6) at (1,-1.5) {1};
				\node (7) at (-1.5,-2.25) {0};
				\node (8) at (-0.5,-2.25) {1};
				\node (9) at (0.5,-2.25) {-1};
				\node (10) at (1.5,-2.25) {0};
			\end{scope}
		\end{tikzpicture}
		
		\caption{Vector $t$.}\label{fig-Tn--3vectors}
	\end{figure}
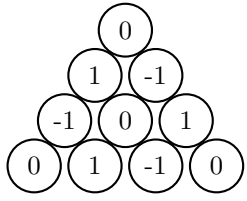
	
	Note that each vector of $\textbf{T}_n$ is obtained placing $t$ in all the possible ways over $\mathbb{T}_n$ and completing the remaining entries with zeros. The ordered pair $(x,y)$ in Equation \eqref{eq-vectors--3-Tn} corresponds to the entry in $\mathbb{T}(n)$ where the first coordinate of $t$ is placed.
The particular case of $\textbf{V}_6$ is presented in Figure \ref{fig-T6-vectors} and this construction is highlighted by the colors.
	
	\begin{figure}[h!]
		\centering
		\begin{tikzpicture}[scale=0.8]
			\begin{scope}[every node/.style={circle,thick,draw,minimum size=0.7cm},every text node part/.style={align=center}]
				\node (1) at (0,0) [circle,draw=blue!50,fill=blue!5] {0};
				\node (2) at (-0.5,-0.75) [circle,draw=blue!50,fill=blue!5] {1};
				\node (3) at (0.5,-0.75) [circle,draw=blue!50,fill=blue!5] {-1};
				\node (4) at (-1,-1.5) [circle,draw=blue!50,fill=blue!5] {-1};
				\node (5) at (0,-1.5) [circle,draw=blue!50,fill=blue!5] {0};
				\node (6) at (1,-1.5) [circle,draw=blue!50,fill=blue!5] {1};
				\node (7) at (-1.5,-2.25) [circle,draw=blue!50,fill=blue!5] {0};
				\node (8) at (-0.5,-2.25) [circle,draw=blue!50,fill=blue!5] {1};
				\node (9) at (0.5,-2.25) [circle,draw=blue!50,fill=blue!5] {-1};
				\node (10) at (1.5,-2.25) [circle,draw=blue!50,fill=blue!5] {0};
				\node (11) at (-2,-3) [circle,draw=blue!50] {0};
				\node (12) at (-1,-3) [circle,draw=blue!50] {0};
				\node (13) at (0,-3) [circle,draw=blue!50] {0};
				\node (14) at (1,-3) [circle,draw=blue!50] {0};
				\node (15) at (2,-3) [circle,draw=blue!50] {0};
				\node (16) at (-2.5,-3.75) [circle,draw=blue!50] {0};
				\node (17) at (-1.5,-3.75) [circle,draw=blue!50] {0};
				\node (18) at (-0.5,-3.75) [circle,draw=blue!50] {0};
				\node (19) at (0.5,-3.75) [circle,draw=blue!50] {0};
				\node (20) at (1.5,-3.75) [circle,draw=blue!50] {0};
				\node (21) at (2.5,-3.75) [circle,draw=blue!50] {0};
			\end{scope}
			\draw (0,-4.8) node {$t_6^{(1,1)}$};
		\end{tikzpicture}	
		\begin{tikzpicture}[scale=0.8]
			\begin{scope}[every node/.style={circle,thick,draw,minimum size=0.7cm},every text node part/.style={align=center}]
				\node (1) at (0,0) [circle,draw=blue!50] {0};
				\node (2) at (-0.5,-0.75) [circle,draw=blue!50,fill=blue!5] {0};
				\node (3) at (0.5,-0.75) [circle,draw=blue!50] {0};
				\node (4) at (-1,-1.5) [circle,draw=blue!50,fill=blue!5] {1};
				\node (5) at (0,-1.5) [circle,draw=blue!50,fill=blue!5] {-1};
				\node (6) at (1,-1.5) [circle,draw=blue!50] {0};
				\node (7) at (-1.5,-2.25) [circle,draw=blue!50,fill=blue!5] {-1};
				\node (8) at (-0.5,-2.25) [circle,draw=blue!50,fill=blue!5] {0};
				\node (9) at (0.5,-2.25) [circle,draw=blue!50,fill=blue!5] {1};
				\node (10) at (1.5,-2.25) [circle,draw=blue!50] {0};
				\node (11) at (-2,-3) [circle,draw=blue!50,fill=blue!5] {0};
				\node (12) at (-1,-3) [circle,draw=blue!50,fill=blue!5] {1};
				\node (13) at (0,-3) [circle,draw=blue!50,fill=blue!5] {-1};
				\node (14) at (1,-3) [circle,draw=blue!50,fill=blue!5] {0};
				\node (15) at (2,-3) [circle,draw=blue!50] {0};
				\node (16) at (-2.5,-3.75) [circle,draw=blue!50] {0};
				\node (17) at (-1.5,-3.75) [circle,draw=blue!50] {0};
				\node (18) at (-0.5,-3.75) [circle,draw=blue!50] {0};
				\node (19) at (0.5,-3.75) [circle,draw=blue!50] {0};
				\node (20) at (1.5,-3.75) [circle,draw=blue!50] {0};
				\node (21) at (2.5,-3.75) [circle,draw=blue!50] {0};
			\end{scope}
			\draw (0,-4.8) node {$t_6^{(2,1)}$};
		\end{tikzpicture}
		\begin{tikzpicture}[scale=0.8]
			\begin{scope}[every node/.style={circle,thick,draw,minimum size=0.7cm},every text node part/.style={align=center}]
				\node (1) at (0,0) [circle,draw=blue!50] {0};
				\node (2) at (-0.5,-0.75) [circle,draw=blue!50] {0};
				\node (3) at (0.5,-0.75) [circle,draw=blue!50,fill=blue!5] {0};
				\node (4) at (-1,-1.5) [circle,draw=blue!50] {0};
				\node (5) at (0,-1.5) [circle,draw=blue!50,fill=blue!5] {1};
				\node (6) at (1,-1.5) [circle,draw=blue!50,fill=blue!5] {-1};
				\node (7) at (-1.5,-2.25) [circle,draw=blue!50] {0};
				\node (8) at (-0.5,-2.25) [circle,draw=blue!50,fill=blue!5] {-1};
				\node (9) at (0.5,-2.25) [circle,draw=blue!50,fill=blue!5] {0};
				\node (10) at (1.5,-2.25) [circle,draw=blue!50,fill=blue!5] {1};
				\node (11) at (-2,-3) [circle,draw=blue!50] {0};
				\node (12) at (-1,-3) [circle,draw=blue!50,fill=blue!5] {0};
				\node (13) at (0,-3) [circle,draw=blue!50,fill=blue!5] {1};
				\node (14) at (1,-3) [circle,draw=blue!50,fill=blue!5] {-1};
				\node (15) at (2,-3) [circle,draw=blue!50,fill=blue!5] {0};
				\node (16) at (-2.5,-3.75) [circle,draw=blue!50] {0};
				\node (17) at (-1.5,-3.75) [circle,draw=blue!50] {0};
				\node (18) at (-0.5,-3.75) [circle,draw=blue!50] {0};
				\node (19) at (0.5,-3.75) [circle,draw=blue!50] {0};
				\node (20) at (1.5,-3.75) [circle,draw=blue!50] {0};
				\node (21) at (2.5,-3.75) [circle,draw=blue!50] {0};
			\end{scope}
			\draw (0,-4.8) node {$t_6^{(2,2)}$};
		\end{tikzpicture}
		
		\vspace{0.3cm}
		
		\begin{tikzpicture}[scale=0.8]
			\begin{scope}[every node/.style={circle,thick,draw,minimum size=0.7cm},every text node part/.style={align=center}]
				\node (1) at (0,0) [circle,draw=blue!50] {0};
				\node (2) at (-0.5,-0.75) [circle,draw=blue!50] {0};
				\node (3) at (0.5,-0.75) [circle,draw=blue!50] {0};
				\node (4) at (-1,-1.5) [circle,draw=blue!50,fill=blue!5] {0};
				\node (5) at (0,-1.5) [circle,draw=blue!50] {0};
				\node (6) at (1,-1.5) [circle,draw=blue!50] {0};
				\node (7) at (-1.5,-2.25) [circle,draw=blue!50,fill=blue!5] {1};
				\node (8) at (-0.5,-2.25) [circle,draw=blue!50,fill=blue!5] {-1};
				\node (9) at (0.5,-2.25) [circle,draw=blue!50] {0};
				\node (10) at (1.5,-2.25) [circle,draw=blue!50] {0};
				\node (11) at (-2,-3) [circle,draw=blue!50,fill=blue!5] {-1};
				\node (12) at (-1,-3) [circle,draw=blue!50,fill=blue!5] {0};
				\node (13) at (0,-3) [circle,draw=blue!50,fill=blue!5] {1};
				\node (14) at (1,-3) [circle,draw=blue!50] {0};
				\node (15) at (2,-3) [circle,draw=blue!50] {0};
				\node (16) at (-2.5,-3.75) [circle,draw=blue!50,fill=blue!5] {0};
				\node (17) at (-1.5,-3.75) [circle,draw=blue!50,fill=blue!5] {1};
				\node (18) at (-0.5,-3.75) [circle,draw=blue!50,fill=blue!5] {-1};
				\node (19) at (0.5,-3.75) [circle,draw=blue!50,fill=blue!5] {0};
				\node (20) at (1.5,-3.75) [circle,draw=blue!50] {0};
				\node (21) at (2.5,-3.75) [circle,draw=blue!50] {0};
			\end{scope}
			\draw (0,-4.8) node {$t_6^{(3,1)}$};
		\end{tikzpicture}	
		\begin{tikzpicture}[scale=0.8]
			\begin{scope}[every node/.style={circle,thick,draw,minimum size=0.7cm},every text node part/.style={align=center}]
				\node (1) at (0,0) [circle,draw=blue!50] {0};
				\node (2) at (-0.5,-0.75) [circle,draw=blue!50] {0};
				\node (3) at (0.5,-0.75) [circle,draw=blue!50] {0};
				\node (4) at (-1,-1.5) [circle,draw=blue!50] {0};
				\node (5) at (0,-1.5) [circle,draw=blue!50,fill=blue!5] {0};
				\node (6) at (1,-1.5) [circle,draw=blue!50] {0};
				\node (7) at (-1.5,-2.25) [circle,draw=blue!50] {0};
				\node (8) at (-0.5,-2.25) [circle,draw=blue!50,fill=blue!5] {1};
				\node (9) at (0.5,-2.25) [circle,draw=blue!50,fill=blue!5] {-1};
				\node (10) at (1.5,-2.25) [circle,draw=blue!50] {0};
				\node (11) at (-2,-3) [circle,draw=blue!50] {0};
				\node (12) at (-1,-3) [circle,draw=blue!50,fill=blue!5] {-1};
				\node (13) at (0,-3) [circle,draw=blue!50,fill=blue!5] {0};
				\node (14) at (1,-3) [circle,draw=blue!50,fill=blue!5] {1};
				\node (15) at (2,-3) [circle,draw=blue!50] {0};
				\node (16) at (-2.5,-3.75) [circle,draw=blue!50] {0};
				\node (17) at (-1.5,-3.75) [circle,draw=blue!50,fill=blue!5] {0};
				\node (18) at (-0.5,-3.75) [circle,draw=blue!50,fill=blue!5] {1};
				\node (19) at (0.5,-3.75) [circle,draw=blue!50,fill=blue!5] {-1};
				\node (20) at (1.5,-3.75) [circle,draw=blue!50,fill=blue!5] {0};
				\node (21) at (2.5,-3.75) [circle,draw=blue!50] {0};
			\end{scope}
			\draw (0,-4.8) node {$t_6^{(3,2)}$};
		\end{tikzpicture}
		\begin{tikzpicture}[scale=0.8]
			\begin{scope}[every node/.style={circle,thick,draw,minimum size=0.7cm},every text node part/.style={align=center}]
				\node (1) at (0,0) [circle,draw=blue!50] {0};
				\node (2) at (-0.5,-0.75) [circle,draw=blue!50] {0};
				\node (3) at (0.5,-0.75) [circle,draw=blue!50] {0};
				\node (4) at (-1,-1.5) [circle,draw=blue!50] {0};
				\node (5) at (0,-1.5) [circle,draw=blue!50] {0};
				\node (6) at (1,-1.5) [circle,draw=blue!50,fill=blue!5] {0};
				\node (7) at (-1.5,-2.25) [circle,draw=blue!50] {0};
				\node (8) at (-0.5,-2.25) [circle,draw=blue!50] {0};
				\node (9) at (0.5,-2.25) [circle,draw=blue!50,fill=blue!5] {1};
				\node (10) at (1.5,-2.25) [circle,draw=blue!50,fill=blue!5] {-1};
				\node (11) at (-2,-3) [circle,draw=blue!50] {0};
				\node (12) at (-1,-3) [circle,draw=blue!50] {0};
				\node (13) at (0,-3) [circle,draw=blue!50,fill=blue!5] {-1};
				\node (14) at (1,-3) [circle,draw=blue!50,fill=blue!5] {0};
				\node (15) at (2,-3) [circle,draw=blue!50,fill=blue!5] {1};
				\node (16) at (-2.5,-3.75) [circle,draw=blue!50] {0};
				\node (17) at (-1.5,-3.75) [circle,draw=blue!50] {0};
				\node (18) at (-0.5,-3.75) [circle,draw=blue!50,fill=blue!5] {0};
				\node (19) at (0.5,-3.75) [circle,draw=blue!50,fill=blue!5] {1};
				\node (20) at (1.5,-3.75) [circle,draw=blue!50,fill=blue!5] {-1};
				\node (21) at (2.5,-3.75) [circle,draw=blue!50,fill=blue!5] {0};
			\end{scope}
			\draw (0,-4.8) node {$t_6^{(3,3)}$};
		\end{tikzpicture}
		\caption{Vectors of $\textbf{V}_6$.}\label{fig-T6-vectors}
	\end{figure}
	
	\medskip
	
	\begin{theorem}\label{Th_Least}
		For $n\ge4$, $-3$ is the least eigenvalue of $\mathcal{T}(n)$ with multiplicity $T(n-3)$ and $\mathbf{T}_n$ is a basis for $\mathcal{E}_{\mathcal{T}(n)}(-3)$.
	\end{theorem}
	\begin{proof}
		Applying Theorem~\ref{Main_result_1} to $\mathcal{T}(n)$ considering the edge clique partition where the vertices of each row, column and diagonal of $\mathcal{T}(n)$ is a clique, it follows that each vector of $\mathbf{T}_n$ is an eigenvector of $\mathcal{T}(n)$ associated with $-3$.
		
		Now consider that $\mathbf{T}_n$ is a set of linearly dependent vectors. Then it exists $ \alpha_{(1,1)},\alpha_{(2,1)},\alpha_{(2,2)},$ $\ldots,\alpha_{(n-3,n-4)},\alpha_{(n-3,n-3)}\in\mathbb{R}^{T(n)}$ not simultaneously zero, such that
		
		\begin{equation}\label{eq-proof-leasteigenvalueTn}
			s=\sum_{(i,j)\in[n]\times[1,i]}\alpha_{(i,j)}t_n^{(i,j)}=0.
		\end{equation}
		
		Let $\ell=T(i-1)+j$ be least integer such that $\alpha_{(i,j)}\not=0$. By the construction of $\mathbf{T}_n$, $s_{\ell+1}=\alpha_{(i,j)}$ and so, by the Equation \eqref{eq-proof-leasteigenvalueTn}, $s_{\ell+1}=0=\alpha_{(i,j)}$ and we get a contradiction. Thus, $\mathbf{T}_n$ is set of linearly independent vectors and $\dim(\mathcal{E}_{\mathcal{T}(n)}(-3))\ge|\mathbf{T}_n|=T(n-3)$.
		
		Let $I\subseteq\{(i,j)\mid i\in[2,n-2],j\in[1,i-1]\}$ be the set of indices $(p,q)$ such that $x_{(p,q)}$, the $(T(p-1)+q)$-th component of $X=(x_{(i,j)})\in\mathcal{E}_{\mathcal{T}(n)}(-3)$ are entirely determined by the entries $x_{(i,j)}$ with $i\in[2,n-2]$ and $j\in[i]$.
		
		By Theorem~\ref{Main_result_1}, we know that the sum of the entries of $x$ by rows, column, and diagonal is zero, then we get that $x_{(1,1)}=x_{(n,1)}=x_{(n,n)}=0$ and $$x_{(i,i)}=-\sum_{j=1}^{i-1}x_{(i,j)}\text{, for }i\in[2,n-1]$$ and so $\{(i,j)\mid i\in[2,n-2],j\in[1,i-1]\}\subset I$. Similarly, we can now state that
		\begin{eqnarray*}
			x_{(n-1,1)} &=& -\sum_{i\in\left([n-2]\cup\{n\}\right)} x_{(i,1)} \quad \text{ and}\\
			x_{(n-1,n-1)} &=& -\sum_{i\in\left([n-2]\cup\{n\}\right)} x_{(i,i)},
		\end{eqnarray*}
		and so $\{(n-1,1),(n-1,n-1)\}\subset I$.
		
		Now, considering the parameters $j_1\in[2,n-1]$ and $j_2=[2,n-2]$, we can obtain
		
		\begin{eqnarray*}
			x_{(n,j_1)} &=& -\sum_{x-y=n-j_1\wedge y\not=j_1} x_{((x,y))} \quad \text{ and}\\
			x_{(n-1,j_2)} &=& -\sum_{x+y=n-1+j_2\wedge y\not=j_2} x_{(x,y)},
		\end{eqnarray*}
		and we conclude that the set of remaining entries is included in $I$, that is, $\{(i,j)\mid i\in\{n-1,n\},j\in[2,i-1]\}\subset I$.
		
		Thus $\dim(\mathcal{E}_{\mathcal{T}(n)}(-3))\le|\mathbf{T}_n|=T(n-3)$ and so $\dim(\mathcal{E}_{\mathcal{T}(n)}(-3))=T(n-3)$.
	\end{proof}
	
	\medskip
	
	From this subsection, we conclude that $-3$ is the least eigenvalue of $\mathcal{T}(n)$ with multiplicity $T(n-3)$ and $2n-2$ is its largest eigenvalue which is simple.
	
	\bigskip
	
	In the next two subsections, vectors $u_{n,\lambda},v_{n,\lambda},x_{n,\lambda}$ and $y_{n,\lambda}$ are introduced. There will be a sequence of integer eigenvalues $\lambda$ of $\mathcal{T}(n)$ with the eigenvectors $u_{n,\lambda}$ and $v_{n,\lambda}$ associated with the eigenvalues of the first part of that sequence, and $x_{n,\lambda}$ and the eigenvectors $y_{n,\lambda}$ associated with the eigenvalues of the second part. Separately into two parts, these vectors will be defined and investigated, in order to conclude that the integers of both sequences are eigenvalues of $\mathcal{T}(n)$.
	
	\subsection{First sequence of integer eigenvalues}\label{FirstSequence}
	
	For $n\ge4$ and $\lambda\in\left[-2,\lfloor\frac{n-7}{2}\rfloor\right] \cap \mathbb{Z}$, consider the vectors $u_{n,\lambda}, u_{n,\lambda}^1, u_{n,\lambda}^2$
	such that
	
	\begin{equation}\label{eq-def-vetorU}
		\begin{split}
			u_{n,\lambda}&=u_{n,\lambda}^1+u_{n,\lambda}^2,
		\end{split}
	\end{equation}
	and
	\begin{equation}
		\begin{split}
			\left[u_{n,\lambda}^1\right]_{(i,j)} &= \begin{cases}
				n-6-2\lambda, &\text{if } i=\lambda+3, \\
				-1, &\text{if } i\ge\lambda+3 \text{ and } i-j\le\lambda+2, \\
				0, &\text{otherwise;}
			\end{cases}\\
			\left[u_{n,\lambda}^2\right]_{(i,j)} &= \begin{cases}
				-(n-6-2\lambda), &\text{if } i-j=n-\lambda-3, \\
				1, &\text{if } i\ge n-\lambda-2 \text{ and }i-j\le n-\lambda-4, \\
				0, &\text{otherwise.}
			\end{cases}
		\end{split}
	\end{equation}
	
	A different and simple way to define these vectors is by the sum of $RCD$-vectors, and so
	\begin{equation}
		\begin{split}
			u_{n,\lambda}^1 = &- (D_{n,0} + \cdots + D_{n,\lambda+2}) \\
			&+ (R_{n,1} + \cdots + R_{n,\lambda+3}) \\
			&+ (n-6-2\lambda) R_{n,\lambda+3} \text{ and} \\
			u_{n,\lambda}^2 = &\phantom{-} (R_{n,n} + \cdots + R_{n,n-\lambda-2}) \\
			&- (D_{n,n-1} + \cdots + D_{n,n-\lambda-3}) \\
			&- (n-6-2\lambda) D_{n,n-\lambda-3}.
		\end{split}
	\end{equation}
	
	An example of these vectors is presented in \cite[Fig. 6.8]{Costa2024}, for $n=10$ and some values of $\lambda$.

	



	Note that the sum-vectors of $u_{n,\lambda}$ are the following.
	
	{\small \begin{equation}\allowdisplaybreaks \label{eq-totalsumU}
		\begin{split}
			S^r_{u_{n,\lambda}}(i) &= \begin{cases}
				                       0, &\text{if } i\le \lambda+2 \text{ or } i \ge n-\lambda-2,\\
				                       (\lambda+3)(n-6-2\lambda), &\text{if } i=\lambda+3,\\
				                      -(\lambda+3), &\text{if } \lambda+4 \le i \le n-\lambda-3;\\
			\end{cases} \\
			S^c_{u_{n,\lambda}}(j) &= 0, \qquad \qquad \qquad \qquad \quad \forall j\in[n]; \\
			S^d_{u_{n,\lambda}} (i-j+1) &= \begin{cases}
				0, &\text{if } i-j\le\lambda+2 \text{ or } i-j\ge n-\lambda-2 ,\\
				\lambda+3, &\text{if } \lambda+3 \le i-j\le n-\lambda-4,\\
				-(\lambda+3)(n-6-2\lambda), &\text{if } i-j=n-\lambda-3.\\
			\end{cases}\\
		\end{split}
	\end{equation}}
	
	\medskip
	
	Consider the following lemma.
	
	\begin{lemma}\label{lemma-u}
		For $n\ge4$ and $\lambda\in\left[-2,\lfloor\frac{n-7}{2}\rfloor\right] \cap \mathbb{Z}$, $u_{n,\lambda}$ is an eigenvector of $\mathcal{T}(n)$ associated with $\lambda$.
	\end{lemma}
	\begin{proof}
		Let $A$ be the adjacency matrix of $\mathcal{T}(n)$. Consider the following equation
		
		\begin{equation}\label{eq-proof-vetoru}
			\begin{split}
				Au_{n,\lambda}(i,j) &= \sum_{(p,q)\sim(i,j)}u_{n,\lambda}(p,q)\\
				&= \underbrace{\sum_{q\in[i]} u_{n,\lambda}(i,q)}_\text{(A)} + \underbrace{\sum_{p\in[j,n]} u_{n,\lambda}(p,j)}_{0} + \underbrace{\sum_{p-q=i-j} u_{n,\lambda}(p,q)}_\text{(B)} - 3u_{n,\lambda}(i,j)
			\end{split}
		\end{equation}
		
In \cite[Table 6.2]{Costa2024}, some parts of Equation \eqref{eq-proof-vetoru} are presented. The columns (A)$+$(B) and $(\lambda+3)u_{n,\lambda}(i,j)$~are complemented with the respective values of Equations \eqref{eq-totalsumU} and \eqref{eq-def-vetorU}, respectively. Since they are equal for each condition, the proof is done.
	\end{proof}

	\medskip

	
	Now consider the vectors $ v_{n,\lambda}, v_{n,\lambda}^1, $ $v_{n,\lambda}^2,v_{n,\lambda}^3, v_{n,\lambda}^{3,k}\in\mathbb{R}^{T(n)}$ for some $n\in\mathbb{N}$ and for all $k\in[\min\{\lambda+3,n-7-2\lambda\}]$ such that
	
	\begin{equation}
		\begin{split}
			v_{n,\lambda}&=v_{n,\lambda}^1+v_{n,\lambda}^2+v_{n,\lambda}^3,\\
			v_{n,\lambda}^3&=\sum_{k=1}^{\min\{\lambda+3,n-2\lambda-7\}} v_{n,\lambda}^{3,k}
		\end{split}
	\end{equation}
	and
	\begin{equation}\label{eq-def-vetorV}
		\begin{split}
			\left[v_{n,\lambda}^1\right]_{(i,j)} &= \begin{cases}
				-T(n-7-2\lambda), &\text{if } i=\lambda+3,\\
				n-7-2\lambda, &\text{if } i\ge\lambda+4 \text{ and } j\le\lambda+3 \text{ and } i-j\le n-(\lambda+4),\\
				0, &\text{otherwise;}
			\end{cases}\\
			\left[v_{n,\lambda}^2\right]_{(i,j)} &= \begin{cases}
				-T(n-7-2\lambda), &\text{if } i=\lambda+3,\\
				n-7-2\lambda,  &\text{if } i\ge\lambda+4 \text{ and } j\le n-(\lambda+3) \text{ and } i-j\le\lambda+2,\\
				0, &\text{otherwise;}
			\end{cases}\\
			\left[v_{n,\lambda}^{3,k}\right]_{(i,j)} &= \begin{cases}
				-2, & \text{if } \lambda+4+k\le i \le n-k \text{ and}\\
				& \phantom{if } k+1 \le j \le n-(\lambda+3+k) \text{ and}\\
				& \phantom{if } k\le i-j \le n-(\lambda+4+k),\\
				0,& \text{otherwise.}
			\end{cases}
		\end{split}
	\end{equation}

	
An example of these vectors is presented in Figure \ref{fig-vet-vv1v2}, for $n=10$ and some values of $\lambda$. Other examples of these vectors are presented in \cite[Appendix A]{Costa2024}.
	
\begin{remark}
The vector $v_{n,\lambda}^3$ is easily described by words. Its non-null entries form a hexagonal shape inside $\mathbb{T}_n$, where the width of the hexagon sides are alternately $\ell_1=\lambda+3$ and $\ell_2=n-7-2\lambda$. Note that we are considering that the triangular form presented in $v_{n,\lambda}^3$ for $\lambda=-2$ is a hexagon where the width of its sides is 1 and $\ell_2$.

Each hexagon has $\min\{\ell_1,\ell_2\}$ layers. The more peripheral layer is $-2$ and going to the inner layers, decreasing by 2 units.
		
The $(\lambda+5)^{\text{th}}$ and the $(n-1)^{\text{th}}$ rows are the first and last non-null rows of $v_{n,\lambda}^3$, respectively, and the one with more non-null entries (the ``largest'' row in the hexagon) is the $(n-\lambda-3)^{\text{th}}$.
		
Furthermore, when $n$ is odd and $\lambda=\frac{n-7}{2}$, the vector $v_{n,\lambda}^3$ is undefined and so it is $v_{n,\lambda}$.
\end{remark}
	
In what follows, the sum-vectors of $v_{n,\lambda}^1$, $v_{n,\lambda}^2$ and $v_{n,\lambda}^3$ are presented.

	\begin{landscape}
		\begin{figure}[h!]
			\centering
			\vspace{-0.5cm}
			\includegraphics[width=23cm]{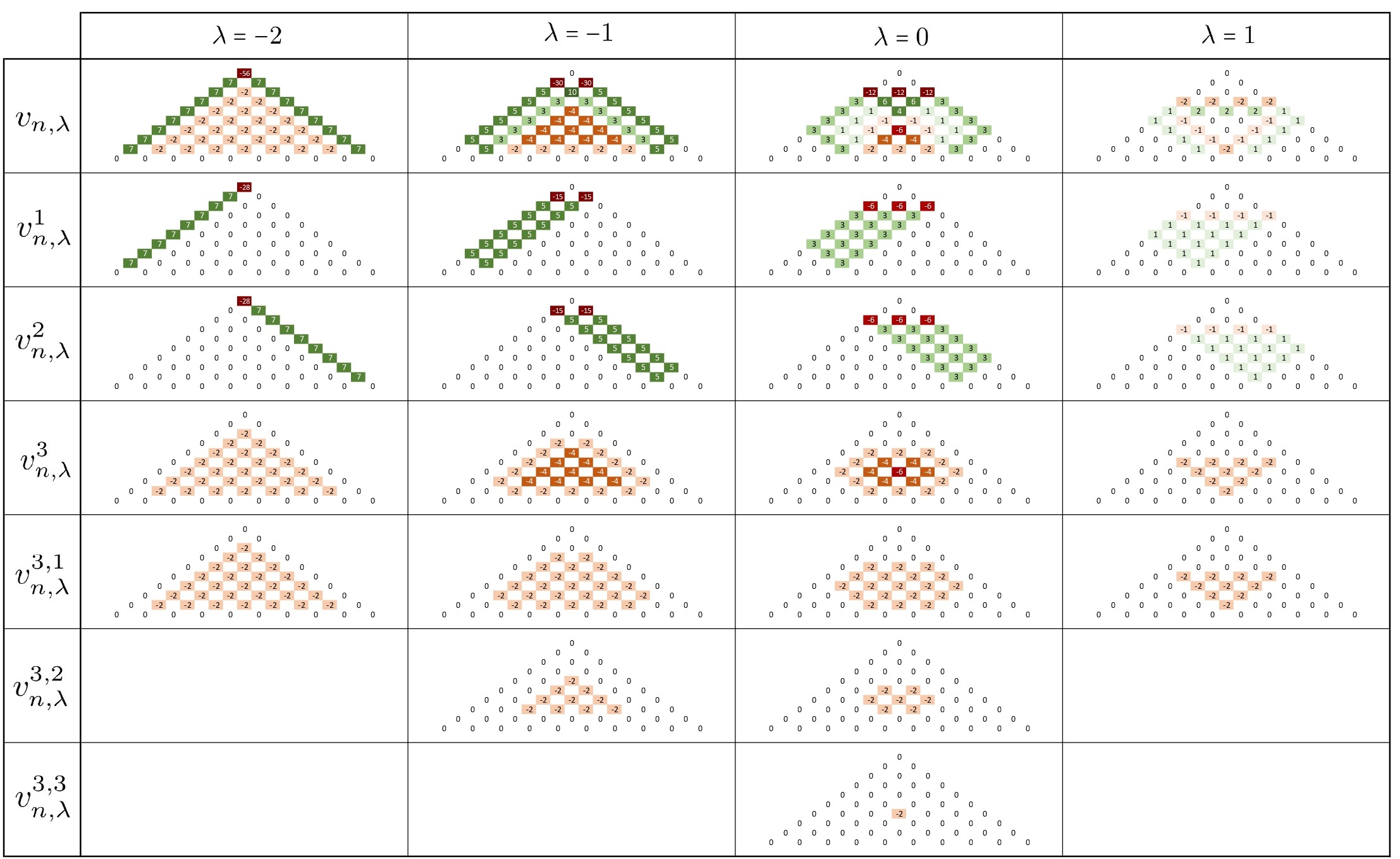}
			\caption{Vectors $v_{n,\lambda}, v_{n,\lambda}^1, v_{n,\lambda}^2$ and $v_{n,\lambda}^3$ for $n=10$ and $\lambda=-2,-1,0,1$.}\label{fig-vet-vv1v2}
		\end{figure}
	\end{landscape}

	\begin{equation}\allowdisplaybreaks
		\begin{split}
			S^r_{v_{n,\lambda}^1}(i) &= \begin{cases}
				0, &\text{if } i\le\lambda+2 \text{ or } i=n,\\
				-T(n-7-2\lambda)(\lambda+3), & \text{if } i=\lambda+3, \\
				(n-7-2\lambda)(\lambda+3), & \text{if } \lambda+4\le i \le n-(\lambda+3), \\
				(n-7-2\lambda)(n-i),& \text{if } n-(\lambda+2)\le i \le n-1; \\
			\end{cases}\\
			S^c_{v_{n,\lambda}^1}(j) &= \begin{cases}
				T(n-7+2\lambda)+(n-7-2\lambda)(j-1), & \text{if } j\le\lambda+3, \\
				0, &\text{ otherwise};
			\end{cases}\\
		\end{split}
	\end{equation}

	\begin{equation}
			S^d_{v_{n,\lambda}^1} (i-j+1) = \begin{cases}
				-T(n-7-2\lambda)+(i-j)(n-7-2\lambda), & \text{if } 0\le i-j \le \lambda+2, \\
				(\lambda+3)(n-7-2\lambda), & \text{if } \lambda+3\le i-j \le n-(\lambda+4), \\
				0, &\text{otherwise;}
			\end{cases}\\
	\end{equation}
	
	\begin{equation}\allowdisplaybreaks
		\begin{split}
			S^r_{v_{n,\lambda}^2}(i) &= \begin{cases}
				0, &\text{if } i\le\lambda+2 \text{ or } i=n,\\
				-T(n-7-2\lambda)(\lambda+3), &\text{if } i=\lambda+3, \\
				(n-7-2\lambda)(\lambda+3), &\text{if } \lambda+4 \le i\le n-(\lambda+3), \\
				(n-7-2\lambda)(n-i),& \text{if } n-(\lambda+2)\le i \le n-1;
			\end{cases}\\
			S^c_{v_{n,\lambda}^2}(j) &= \begin{cases}
				-T(n-7-2\lambda)+(j-1)(n-7+2\lambda), & \text{if } j\le\lambda+3, \\
				(\lambda+3)(n-7-2\lambda), & \text{if } \lambda+4 \le j \le n-(\lambda+3),\\
				0, &\text{ otherwise};
			\end{cases}\\
			S^d_{v_{n,\lambda}^2} (i-j+1) &= \begin{cases}
				T(n-7-2\lambda)+(i-j)(n-7-2\lambda), & \text{if } 0\le i-j \le\lambda+2,\\
				0, &\text{otherwise.}
			\end{cases}\\
		\end{split}
	\end{equation}
	
	Note that, since $v_{n,\lambda}^3=\left(v_{n,\lambda}^3\right)^-=\left(v_{n,\lambda}^3\right)^+$, $S^r_{v_{n,\lambda}^3} (i)= S^c_{v_{n,\lambda}^3} (n-i+1) =S^d_{v_{n,\lambda}^3} (n-i)$, and then
	
	\begin{equation}\allowdisplaybreaks
		\begin{split}
			S^r_{v_{n,\lambda}^3}(i) &= \begin{cases}
				0, &\text{if } i\le\lambda+4 \text{ or } i=n,\\
				-2(\lambda+3)(i-(\lambda+4)), &\text{if } \lambda+5 \le i \le n-(\lambda+3),\\
				-2(n-i)(n-7-2\lambda
				), & \text{if } n-(\lambda+2) \le i \le n-1;
			\end{cases}\\
			S^c_{v_{n,\lambda}^3}(j) &= \begin{cases}
				0, &\text{if } j=1 \text{ or } j\ge n-(\lambda+3),\\
				-2(n-7-2\lambda)(j-1), &\text{if } 2 \le j \le \lambda+3, \\
				-2(\lambda+3)(n-j-(\lambda+3)), & \text{if } \lambda+4 \le j \le n-(\lambda+4);
			\end{cases}\\
			S^d_{v_{n,\lambda}^3}(i-j+1) &= \begin{cases}
				0, &\text{if } i-j=0 \text{ or } i-j\ge n-(\lambda+3),\\
				-2(i-j)(n-7-2\lambda), &\text{if } 1 \le i-j \le \lambda+2,\\
				-2(\lambda+3)(n-(i-j)-(\lambda+4)), & \text{if } \lambda+3 \le i-j \le n-(\lambda+4).
			\end{cases}
		\end{split}
	\end{equation}

	Now we are in position to define the sum-vectors of $v_{n,\lambda}$.
	
	\begin{equation}\allowdisplaybreaks \label{eq-totalsumV}
		\begin{split}
			S^r_{v_{n,\lambda}}(i) &= \begin{cases}
				0, &\text{if } i\le\lambda+2 \text{ or } n-(\lambda+2)\le i,\\
				-2T(n-7-2\lambda)(\lambda+3), &\text{if } i=\lambda+3, \\
				2(\lambda+3)(n-3-\lambda-i), &\text{if } \lambda+4 \le i \le n-(\lambda+3);
			\end{cases}\\
			S^c_{v_{n,\lambda}}(j) &= \begin{cases}
				0, &  \text{if } j\le \lambda+3 \text{ or } j\ge n-(\lambda+2),\\
				(\lambda+3)(2j-n-1), & \text{if } \lambda+4\le j\le n-(\lambda+3);\\
			\end{cases} \\
			S^d_{v_{n,\lambda}}(i-j+1) &= \begin{cases}
				0, & \text{if } i-j\le \lambda+2 \text{ or } i-j\ge n-(\lambda+3),\\
				(\lambda+3)(2(i-j)+1-n), & \text{if } \lambda+3 \le i-j \le n-(\lambda+4).
			\end{cases}
		\end{split}
	\end{equation}


	Consider the following Lemma.
	
	\begin{lemma}\label{lemma-v}
		For $n\ge4$ and $\lambda\in\left[-2,\lfloor\frac{n-7}{2}\rfloor\right]$, $v_{n,\lambda}$ in an eigenvector of $\mathcal{T}(n)$ associated with $\lambda$.
	\end{lemma}
	\begin{proof}
		Let $A$ be the adjacency matrix of $\mathcal{T}(n)$. Consider the following equation
		
		\begin{equation}\label{eq-proof-vetorv}
			\begin{split}
				Av_{n,\lambda}(i,j) &= \sum_{(p,q)\sim(i,j)}v_{n,\lambda}(p,q)\\
				&= \underbrace{\sum_{q\in[i]} v_{n,\lambda}(i,q)}_\text{(A)} + \underbrace{\sum_{p\in[j,n]} v_{n,\lambda}(p,j)}_{(B)} + \underbrace{\sum_{p-q=i-j} v_{n,\lambda}(p,q)}_\text{(C)} - 3v_{n,\lambda}(i,j)
			\end{split}
		\end{equation}
		
Taking into account \eqref{eq-totalsumV} and \eqref{eq-def-vetorV}, it follows that (A)$+$(B)$+$(C) and $(\lambda+3)(u_{n,\lambda}^1(p,q)+u_{n,\lambda}^2(p,q)+u_{n,\lambda}^3(p,q))=(\lambda+3) u_{n,\lambda}(p,q)$ are equal for each of the following conditions.

		\begin{itemize}
			\item[(C1)] $i\le\lambda+2$ and $j\le\lambda+3$ and $i-j\le\lambda+2$;
			\item[(C2)] $i=\lambda+3$ and $j\le\lambda+3$ and $i-j\le\lambda+2$;
			\item[(C3)] $\lambda+4\le i \le n-(\lambda+3)$ and $j\le\lambda+3$ and $i-j\le\lambda+2$;
			\item[(C4)] $\lambda+4\le i \le n-(\lambda+3)$ and $j\le\lambda+3$ and $\lambda+3\le i-j \le n-(\lambda+4)$;
			\item[(C5)] $\lambda+4\le i \le n-(\lambda+3)$ and $\lambda+4\le j \le n-(\lambda+3)$ and $i-j\le\lambda+2$;
			\item[(C6)] $\lambda+4\le i \le n-(\lambda+3)$ and $\lambda+4\le j \le n-(\lambda+3)$ and $\lambda+3\le i-j \le n-(\lambda+4)$;
			\item[(C7)] $n-(\lambda+2)\le i$ and $j\le\lambda+3$ and $\lambda+3\le i-j \le n-(\lambda+4)$;
			\item[(C8)] $n-(\lambda+2)\le i$ and $j\le\lambda+3$ and $n-(\lambda+3)\le i-j$;
			\item[(C9)] $n-(\lambda+2)\le i$ and $\lambda+4\le j \le n-(\lambda+3)$ and $i-j\le\lambda+2$;
			\item[(C10)] $n-(\lambda+2)\le i$ and $\lambda+4\le j \le n-(\lambda+3)$ and $\lambda+3\le i-j \le n-(\lambda+4)$;
			\item[(C11)] $n-(\lambda+2)\le i$ and $n-(\lambda+2)\le j$ and $i-j\le\lambda+2$.			
		\end{itemize}
Therefore, the proof is done.

\end{proof}

	\begin{lemma}\label{lemma-uuv}
		$u_{n,\lambda}, u_{n,\lambda}^-$ and $v_{n,\lambda}$ are linearly independent.
	\end{lemma}
	\begin{proof}
		Let $\alpha,\beta,\gamma\in\mathbb{R}$ and $\mathbf{0}$ the null-vector with $T(n)$ components such that
		
		\begin{equation}\label{eq-abcTn}
			\alpha u_{n,\lambda} + \beta u_{n,\lambda}^- + \gamma v_{n,\lambda} = \mathbf{0}.
		\end{equation}
		
		Considering the Equation \eqref{eq-abcTn} and since, for $i-j = n-\lambda-3$, $u_{n,\lambda}(i,j)\not=0$ and $u_{n,\lambda}^-(i,j) = v_{n,\lambda}(i,j) = 0$, then $\alpha=0$. On the other hand, since $u_{n,\lambda}^-(\lfloor\frac{n}{2}\rfloor,\lfloor\frac{n}{2}\rfloor) = 0$ and $v_{n,\lambda}(\lfloor\frac{n}{2}\rfloor,\lfloor\frac{n}{2}\rfloor) \not= 0$, then $\gamma=0$. Finally $\beta$ has to be zero and so these vectors are linearly independent.
	\end{proof}
	
	\begin{remark}
		Note that the vectors $u_{n,\lambda}$, $u_{n,\lambda}^-$ and $u_{n,\lambda}^+$ are not linearly independent, since the sum of them is the null vector. However, if we consider just two of them, like in the previous Lemma, they are.
	\end{remark}
	
	\bigskip
	
	In this section, we proved that $\lambda\in\{-2,\ldots,\lfloor\frac{n-7}{2}\rfloor\}$ are eigenvalues of $\mathcal{T}(n)$. Three linearly independent eigenvectors associated with $\lambda$ were introduced, $u_{n,\lambda}, u_{n,\lambda}^-$ and $v_{n,\lambda}$, and so its multiplicity is at least $3$, except for $\lambda=\frac{n-7}{2}$, when $n$ is odd, whose multiplicity is at least $2$.
	
	\subsection{Second sequence of integer eigenvalues}\label{SecondSequence}
	
	For $n\ge4$ and $\lambda\in\left[\lceil\frac{n-4}{2}\rceil,n-3\right] \cap \mathbb{Z}$, consider the vectors $x_{n,\lambda}, x_{n,\lambda}^1, x_{n,\lambda}^2$ such that
	\begin{equation}\label{eq-def-vetorX}
		\begin{split}
			x_{n,\lambda}&=x_{n,\lambda}^1 + x_{n,\lambda}^2
		\end{split}
	\end{equation}
	with
	\begin{equation}
		\begin{split}
			\left[x_{n,\lambda}^1\right]_{(i,j)} &= \begin{cases}
				2\lambda-n+6, &\text{if } i-j=1-(n-\lambda-2)\\
				1, &\text{if } i-j<1-(n-\lambda-2) \text{ and } j\le n-\lambda+1\\
				0, &\text{otherwise;}
			\end{cases}\\
			\left[x_{n,\lambda}^2\right]_{(i,j)} &= \begin{cases}
				-(2\lambda-n+6), &\text{if } j=n-\lambda-2\\
				-1, &\text{if } j<n-\lambda-2 \text{ and } i-j\le \lambda+2\\
				0, &\text{otherwise;}
			\end{cases}\\
		\end{split}
	\end{equation}
	
	\begin{figure}[h!]
		\centering
		\includegraphics[width=13cm]{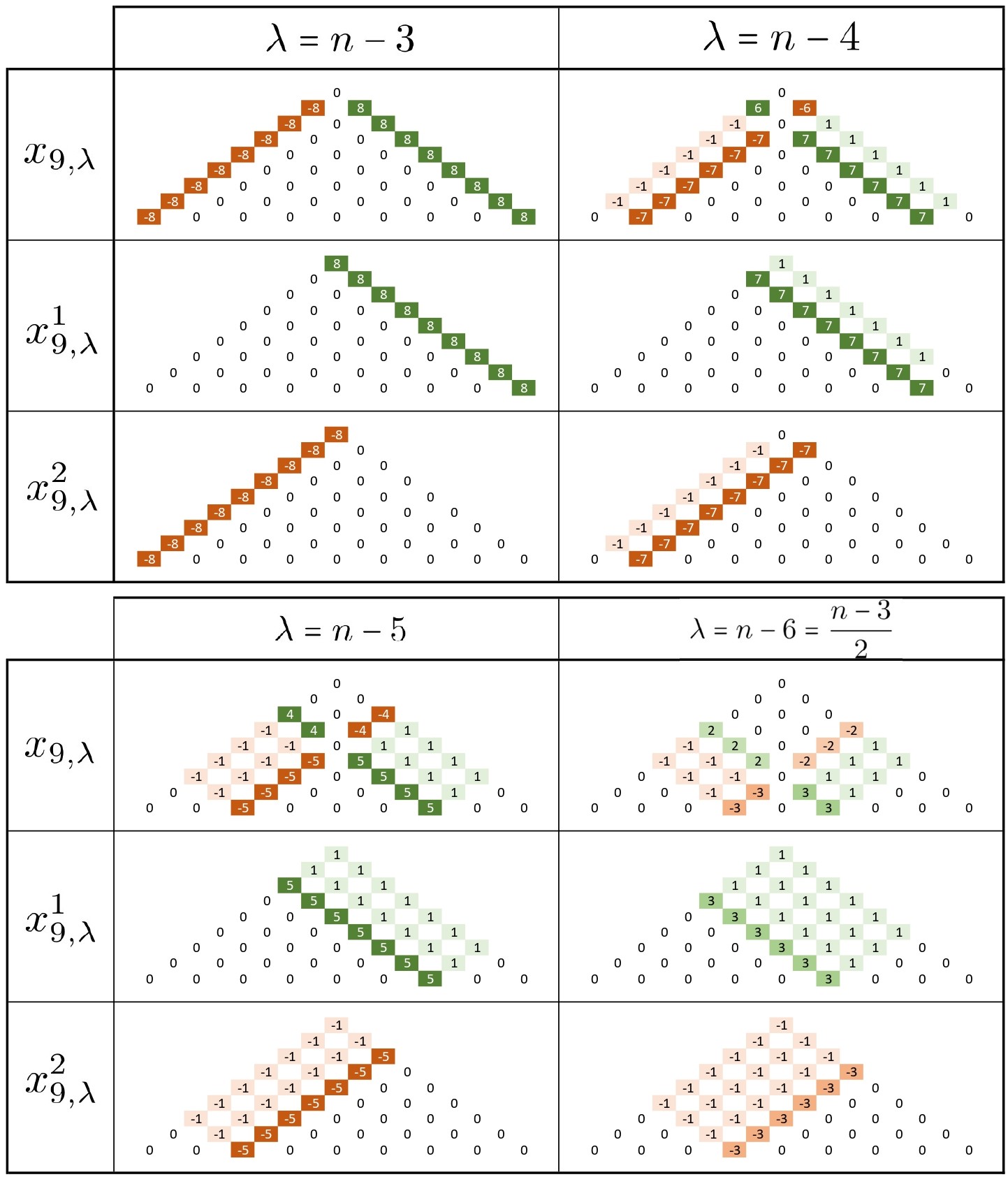}
		\caption{Vectors $x_{n,\lambda}, x_{n,\lambda}^1$ and $x_{n,\lambda}^2$ for $n=9$ and $\lambda=n-3,n-4,n-5,n-6$.}\label{fig-vet-xx1x2}
	\end{figure}
	
	A different and simple way to define these vectors is by the sum of $RCD$-vectors, and so
	\begin{equation}
		\begin{split}
			x_{n,\lambda}^1 = &\phantom{-} (D_{n,n-\lambda-4} + \ldots + D_{n,0})\\
			&- (C_{n,n} + \ldots + C_{n,\lambda+3})\\
			&+ (2\lambda-n+6) D_{n,n-\lambda-3} \text{ and}\\
			x_{n,\lambda}^2 = &- (C_{n,n-\lambda-3} + \ldots + C_{n,1})\\
			&+ (D_{n,n-1} + \ldots + D_{n,\lambda+3
			})\\
			&- (2\lambda-n+6) C_{n,n-\lambda-2}.
		\end{split}
	\end{equation}
	
	An example of these vectors is presented in Figure \ref{fig-vet-xx1x2}, for $n=10$ and some values of $\lambda$. Other examples of these vectors are presented in \cite[Appendix B]{Costa2024}.
	
	\medskip

	Note that the sum-vectors of $x_{n,\lambda}$ are the following.
	
	{\small \begin{equation}\label{eq-totalsumX}\allowdisplaybreaks
		\begin{split}
			S^r_{x_{n,\lambda}}(i) &= 0\\
			S^c_{x_{n,\lambda}}(j) &= \begin{cases}
				0, & 
                \text{if } j\le n-\lambda-3 \text{ or } j\ge\lambda+3, \\
				-(2\lambda-n+5) \times \\
                \times ((2\lambda-n-6)+(n-\lambda-3)), &\text{if } j=n-\lambda-2,\\
				(2\lambda-n+6)+(n-\lambda-3), & 
                \text{if } n-\lambda-1\le j \le \lambda+2;\\
			\end{cases}\\
			S^d_{x_{n,\lambda}}(i-j+1) &= \begin{cases}
				0, & 
                \text{if } i-j\le n-\lambda-5 \text{ or } i-j\ge\lambda+3, \\
				(2\lambda-n+5) \times & \\
                \times ((2\lambda-n-6)+(n-\lambda-3)), &\text{if } i-j=n-\lambda-4,\\
				-(2\lambda-n+6)-(n-\lambda-3), & 
                \text{if } n-\lambda-3\le i-j \le<\lambda+2.\\
			\end{cases}\\
		\end{split}
	\end{equation}}
	
	Consider the following Lemma.
	
	\begin{lemma}\label{lemma-x}
		For $n\ge4$ and $\lambda\in\left[\lceil\frac{n-4}{2}\rceil,n-3\right] \cap \mathbb{Z}$, $x_{n,\lambda}$ is an eigenvector of $\mathcal{T}(n)$ associated with $\lambda$.
	\end{lemma}

	\begin{proof}
		Let $A$ be the adjacency matrix of $\mathcal{T}(n)$. Consider the following equation
		
		{\small \begin{equation}\label{eq-proof-vetorx}
			\begin{split}
				Ax_{n,\lambda}(i,j) &= \sum_{(p,q)\sim(i,j)}x_{n,\lambda}(p,q)\\
				&= \underbrace{\sum_{q\in[i]} x_{n,\lambda}(i,q)}_{=0} + \underbrace{\sum_{p\in[j,n]} x_{n,\lambda}(p,j)}_\text{(A)} + \underbrace{\sum_{p-q=i-j} x_{n,\lambda}(p,q)}_\text{(B)} - 3x_{n,\lambda}(i,j)
			\end{split}
		\end{equation}}
In \cite[Table 6.4]{Costa2024}, some parts of Equation \eqref{eq-proof-vetorx} are presented. The columns $(A)$+$(B)$ and $(\lambda+3)x_{n,\lambda}(i,j)$ are complemented with the values of \eqref{eq-totalsumX} and \eqref{eq-def-vetorX}, respectively. Since they are equal for each condition, the proof is done.
\end{proof}

Now consider the vectors $y_{n,\lambda},y_{n,\lambda}^1,y_{n,\lambda}^2,y_{n,\lambda}^{2,1},y_{n,\lambda}^{2,2},y_{n,\lambda}^{2,3},y_{n,\lambda}^{2,4},y_{n,\lambda}^3,y_{n,\lambda}^{3,k}\in\mathbb{R}^{T(n)}$ for some $n\in\mathbb{N}$ and for all $k\in[\min\{n-\lambda-2,-n+2\lambda+4\}]$ such that
\begin{equation}\label{eq-def-vetorY}
		\begin{split}
			y_{n,\lambda}&=y_{n,\lambda}^1 + y_{n,\lambda}^2 + y_{n,\lambda}^3,\\
			y_{n,\lambda}^2&=y_{n,\lambda}^{2,1}+y_{n,\lambda}^{2,2}+y_{n,\lambda}^{2,3}+y_{n,\lambda}^{2,4}\text{ and}\\
			y_{n,\lambda}^3&=\sum_{k=1}^{\min\{n-\lambda-2,-n+2\lambda+4\}} y_{n,\lambda}^{3,k}
		\end{split}
\end{equation}
with
\begin{equation}\label{eq-def-y1}
		\left[y_{n,\lambda}^1\right]_{(i,j)} = \begin{cases}
			0, &\text{if } j\le n-\lambda-3 \text{ or } j\ge\lambda+3, \\
			n-2\lambda-6, &\text{if } (-n+2\lambda+5\le j \le \lambda+2) \text{ or } \\
			&\big((n-\lambda-2\le j \le -n+2\lambda+4) \text{ and } \\
			& (i\ge\lambda+3\text{ or } i-j\le n-\lambda-3)\big), \\
			2(n-2\lambda-5), &\text{if } i\le \lambda+2 \text{ and } j\ge n-\lambda-2 \text{ and } i-j\ge n-\lambda-2;
		\end{cases}\\
\end{equation}
\begin{equation}
		\begin{split}\allowdisplaybreaks
			\left[y_{n,\lambda}^{2,1}\right]_{(i,j)} &=\begin{cases}
				2T(n-2\lambda-6), & \text{if } j=n-\lambda-2, \\
				0, &\text{otherwise;}
			\end{cases}\\
			\left[y_{n,\lambda}^{2,2}\right]_{(i,j)} &= \begin{cases}
				-n+2\lambda+4, & \text{if } i\ge n-\lambda-2, \\
				&j\le n-\lambda-3, \text{ and}\\
				& i-j \le n-\lambda-3, \\
				0, &\text{otherwise;}
			\end{cases}\\
			\left[y_{n,\lambda}^{2,3}\right]_{(i,j)} &= \begin{cases}
				-n+2\lambda+4, & \text{if } i\ge \lambda+3, \\
				&j\le n-\lambda-3, \text{ and}\\
				& i-j \le\lambda+2, \\
				0, &\text{otherwise;}
			\end{cases}\\
			\left[y_{n,\lambda}^{2,4}\right]_{(i,j)} &= \begin{cases}
				-n+2\lambda+4, & \text{if } i\ge \lambda+3, \\
				& n-\lambda-2 \le j\le \lambda+2, \text{ and} \\
				& i-j \le n-\lambda-3, \\
				0, &\text{otherwise;}
			\end{cases}\\\end{split}
\end{equation}
\begin{equation}
		\left[y_{n,\lambda}^{3,k}\right]_{(i,j)} =\begin{cases}
			                                        -2, & \text{if } n-\lambda-2+k\le i \le n-k, \\
			                                            & k \le j \le \lambda+2-k, \text{ and}\\
                                                        & k\le i-j \le n-k-1, \\
			                                         0, &\text{otherwise.}
		                                            \end{cases}
\end{equation}

An example of these vectors is presented in \cite[Fig. 6.12]{Costa2024}, for $n=10$ and some values of $\lambda$. Other examples of these vectors are presented in \cite[Appendix B]{Costa2024}.
	
	
\begin{remark}
The vector $y_{n,\lambda}^2$ is the sum of four vectors:
\begin{itemize}\setlength\itemsep{0.1em}
\item  $y_{n,\lambda}^{2,1}$ - a vector whose $n-\lambda-2^{\text{th}}$ column is $2T(n-2\lambda-6)$ and the remaining entries are zero;
\item $y_{n,\lambda}^{2,2}$ whose non null components are $-n+2\lambda+4$ and they form a triangular shape limited by conditions $i=n-\lambda-2$, $j=n-\lambda-3$ and $i-j=n-\lambda-3$;
\item $y_{n,\lambda}^{2,3}$ whose non null components are $-n+2\lambda+4$ and they form a triangular shape limited by conditions $i=\lambda+3$, $j=n-\lambda-3$ and $i-j=\lambda+2$;
\item $y_{n,\lambda}^{2,4}$ whose non null components are $-n+2\lambda+4$ and, when $2(n-\lambda-3)<\lambda+3$, they form a triangular shape limited by conditions $i=\lambda+3$, $j=\lambda+2$ and $i-j=n-\lambda-3$, when $2(n-\lambda-3)\ge\lambda+3$, they are limited by conditions $i=\lambda+3$, $j=n-\lambda-2$, $j=\lambda+2$ and $i-j=n-\lambda-3$.
\end{itemize}
Note that the $2(n-\lambda-3)^{\text{th}}$ and the $\lambda+3^{\text{th}}$ rows are the last row of $y_{n,\lambda}^{2,2}$ and first row of $y_{n,\lambda}^{2,3}$ with non-null components, respectively. Furthermore, when $2n-2\lambda-6<\lambda+2$, none of these four vectors have non-zero entries in common. When $2n-2\lambda-6=\lambda+2$ the vectors $y_{n,\lambda}^{2,1}$ and $y_{n,\lambda}^{2,4}$ have one non-null entry in common. And finally, when $2n-2\lambda-6\ge\lambda+3$, the vectors $y_{n,\lambda}^{2,1}$ and $y_{n,\lambda}^{2,4}$, and the vectors $y_{n,\lambda}^{2,2}$ and $y_{n,\lambda}^{2,3}$ have some non-null components in common.
\end{remark}
%
	
\begin{remark}
When the first column (and row) of $v_{n+1,n-\lambda-5}^3$ are removed, we obtained $y_{n,\lambda}^3$.
		
Note that, as in the vector $v_{n+1,n-\lambda-5}^3$, the non-nulls components of $y_{n,\lambda}^3$ form a hexagonal shape inside $\mathbb{T}(n)$, where the hexagon sides are alternately $\ell_1=n-\lambda-2$ and $\ell_2=-n+2\lambda+4$.
		
Also note that, when $n$ is even and $\lambda=\frac{n-4}{2}$, $\ell_2=0$. Thus the vector $y_{n,\lambda}^3$ is undefined and so is $y_{n,\lambda}$.
\end{remark}
	
In what follows, we will compute the sum-vectors of $y_{n,\lambda}^1$, $y_{n,\lambda}^2$ and $y_{n,\lambda}^3$ by row, column, and diagonal. Since these vectors have significant changes for different values of $n$ and $\lambda$, we will divide these sums into different cases when necessary.
	
	\medskip
	
	If $2n-2\lambda-5<\lambda+2$, the sum-vectors of $y_{n,\lambda}^1$ are the following.
	\begin{equation}\label{eq-SUMY-1}\allowdisplaybreaks
		\begin{split}
			S^r_{y_{n,\lambda}^1}(i) &= \begin{cases}
				0, & \text{if } i\le n-\lambda-3, \\
				(n-2\lambda-6)(i-n+3+\lambda), & \text{if } n-\lambda-2\le i\le 2n-2\lambda-5, \\
				(n-2\lambda-5)(-3n+3\lambda+2i+8)-(n-\lambda-2), & \text{if } 2n-2\lambda-4\le i\le \lambda+2, \\
				-2T(n-2\lambda-6), & \text{if } \lambda+3\le i ;\\
			\end{cases}\\
			S^c_{y_{n,\lambda}^1}(j) &= \begin{cases}
				0&\text{if } j\le n-\lambda-3 \text{ or } j\ge\lambda+3, \\
				2(n-2\lambda-5)(\lambda-j+3)-2(n-\lambda-2) &\text{if }n-\lambda-2\le j\le -n+2\lambda+4, \\
				(n-2\lambda-6)(n-j+1), &\text{if } -n+2\lambda+5\le j \le \lambda+2;
			\end{cases}\\
			S^d_{y_{n,\lambda}^1}(i-j+1) &= \begin{cases}
				-2T(n-2\lambda-6), &\text{if }i-j\le n-\lambda-3, \\
				(n-2\lambda-5)(-n+3\lambda-2(i-j)+8)-(n-\lambda-2), &\text{if }n-\lambda-2\le i-j \le -n+2\lambda+4, \\
				(n-2\lambda-6)(\lambda+3-(i-j)), &\text{if } -n+2\lambda+5\le i-j\le\lambda+2, \\
				0, &\text{if } i-j\ge\lambda+3.\\
			\end{cases}
		\end{split}
	\end{equation}
	
	Otherwise, if $2n-2\lambda-5\ge\lambda+2$, the sum-vectors of $y_{n,\lambda}^1$ are the following.
	\begin{equation}\label{eq-SUMY-2}\allowdisplaybreaks
		\begin{split}
			S^r_{y_{n,\lambda}^1}(i) &= \begin{cases}
				0, & \text{if } i\le n-\lambda-3, \\
				(n-2\lambda-6)(i-n+3+\lambda), & \text{if } n-\lambda-2\le i\le \lambda+2, \\
				-2T(n-2\lambda-6), & \text{if } \lambda+3\le i ;\\
			\end{cases}\\
			S^c_{y_{n,\lambda}^1}(j) &= \begin{cases}
				0&\text{if } j\le n-\lambda-3 \text{ or } j\ge\lambda+3, \\
				(n-2\lambda-6)(n-j+1), &\text{if } n-\lambda-2\le j \le \lambda+2;
			\end{cases}\\
			S^d_{y_{n,\lambda}^1}(i-j+1) &= \begin{cases}
				-2T(n-2\lambda-6), &\text{if }i-j\le n-\lambda-3, \\
				(n-2\lambda-6)(\lambda+3-(i-j)), &\text{if } n-\lambda-2\le i-j\le\lambda+2, \\
				0, &\text{if } i-j\ge\lambda+3.\\
			\end{cases}
		\end{split}
	\end{equation}
	
	If $2n-2\lambda-5<\lambda+3$, the sum-vectors of $y_{n,\lambda}^2$ are
	\begin{equation}\label{eq-SUMY-3}\allowdisplaybreaks
		\begin{split}
			S^r_{y_{n,\lambda}^2}(i) &= \begin{cases}
				0, & \text{if } i\le n-\lambda-3, \\
				2T(n-2\lambda-6)+(-n+2\lambda+4)(2n-2\lambda-5-i), &\text{if }  n-\lambda-2\le i \le 2n-2\lambda-5, \\
				2T(n-2\lambda-6), & \text{if } 2n-2\lambda-4\le i \le \lambda+2, \\
				2T(n-2\lambda-6)+2(-n+2\lambda+4)(n-i), &\text{if } \lambda+3 \le i;
			\end{cases}\\
			S^c_{y_{n,\lambda}^2}(j) &= \begin{cases}
				2j(-n+2\lambda+4), &\text{if } j \le n-\lambda-3, \\
				2(\lambda+3)T(n-2\lambda-6), &\text{if } j=n-\lambda-2, \\
				0, &\text{if } n-\lambda-1\le j \le -n+2\lambda+4, \\
				(-n+2\lambda+4)(j+n-2\lambda-5) &\text{if } -n+2\lambda+5\le j \le\lambda+2, \\
				0, &\text{if } j\ge\lambda+3;
			\end{cases}\\
			S^d_{y_{n,\lambda}^2}(i-j+1) &= \begin{cases}
				2T(n-2\lambda-6)+2(i-j)(-n+2\lambda+4), &\text{if } i-j\le n-\lambda-3, \\
				2T(n-2\lambda-6), &\text{if } n-\lambda-2\le i-j\le -n+2\lambda+4, \\
				2T(n-2\lambda-6)+(i-j+n-2\lambda-5)(-n+2\lambda+4), &\text{if } -n+2\lambda+5\le i-j \le\lambda+2, \\
				0, &\text{if }i-j\ge\lambda+3.
			\end{cases}
		\end{split}
	\end{equation}
	
	Otherwise, if $2n-2\lambda-5\ge\lambda+3$, they are
	
	\begin{equation}\label{eq-SUMY-4}\allowdisplaybreaks
		\begin{split}
			S^r_{y_{n,\lambda}^2}(i) &= \begin{cases}
				0, & \text{if } i\le n-\lambda-3, \\
				2T(n-2\lambda-6)+(-n+2\lambda+4)(2n-2\lambda-5-i), &\text{if } n-\lambda-2\le i \le \lambda+2, \\
				2T(n-2\lambda-6)+2(-n+2\lambda+4)(n-i), &\text{if } \lambda+3 \le i;
			\end{cases}\\
			S^c_{y_{n,\lambda}^2}(j) &= \begin{cases}
				2(-n+2\lambda+4)j, &\text{if } j \le n-\lambda-3, \\
				2(\lambda+3)T(n-2\lambda-6)+(-n+2\lambda+4)(2n-3\lambda-7), &\text{if } j=n-\lambda-2, \\
				(-n+2\lambda+4)(j+n-2\lambda-5), &\text{if } n-\lambda-1\le j\le \lambda+2, \\
				0, &\text{if } j\ge\lambda+3;		
			\end{cases}\\
			S^d_{y_{n,\lambda}^2}(i-j+1) &= \begin{cases}
				2T(n-2\lambda-6)+2(i-j)(-n+2\lambda+4), &\text{if } i-j\le n-\lambda-3, \\
				2T(n-2\lambda-6)+(i-j+n-2\lambda-5)(-n+2\lambda+4), &\text{if } n-\lambda-2\le i-j \le \lambda+2, \\
				0, &\text{if }i-j\ge\lambda+3.
			\end{cases}
		\end{split}
	\end{equation}
	
	Finally, the sum-vectors of $y_{n,\lambda}^3$ are
	
	\begin{equation}\label{eq-SUMY-5}\allowdisplaybreaks
		\begin{split}		
			S^r_{y_{n,\lambda}^3}(i) &= S^r_{v_{n+1,n-\lambda-5}^3}(i+1)=\begin{cases}
				0, &\text{if } i\le n-\lambda-3, \\
				-2(n-\lambda-2)(-n+\lambda+i+2), &\text{if }n-\lambda-2\le i \le \lambda+2, \\
				-2(n-i)(-n+4+2\lambda), &\text{if }\lambda+3\le i; \\
			\end{cases}\\
			S^c_{y_{n,\lambda}^3}(j) &= S^c_{v_{n+1,n-\lambda-5}^3}(j+1)=\begin{cases}
				-2j(-n+2\lambda+4), &\text{if } j\le n-\lambda-3, \\
				-2(n-\lambda-2)(\lambda-j+2), &\text{if } n-\lambda-2\le j\le\lambda+2, \\
				0, &\text{if } j\ge\lambda+3; \\
			\end{cases}\\
			S^d_{y_{n,\lambda}^3}(i-j) &= S^d_{v_{n+1,n-\lambda-5}^3}(i-j) =\begin{cases}
				-2(i-j)(-n+2\lambda+4), & \text{if } i-j\le n-\lambda-3, \\
				-2(n-\lambda-2)(\lambda-(i-j)+2), &\text{if } n-\lambda-2\le i-j\le \lambda+2, \\
				0, &\text{if } i-j\ge\lambda+3.
			\end{cases}\\
		\end{split}
	\end{equation}
	
Now we are in position to define the sum-vectors for $y_{n,\lambda}$. Although the partial sum-vector are different for each case, the sum-vectors of $y_{n,\lambda}$ are always
	
	\begin{equation}\label{eq-sumvector-y}
		\begin{split}
			S_{y_{n,\lambda}}^r(i) &= \begin{cases}
				0, &\text{if } i\le n-\lambda-3\\
				(\lambda+3)(n-2i), &\text{if } n-\lambda-2\le i\le \lambda+2\\
				0, &\text{if } \lambda+3\le i;
			\end{cases}\\
			S_{y_{n,\lambda}}^r(j) &= \begin{cases}
				0, &\text{if } j\le n-\lambda-3\\
				(\lambda+3)(n-2\lambda-5)(n-2\lambda-4), &\text{if } j=n-\lambda-2\\
				-2(\lambda+3)(\lambda+3-j), &\text{if } n-\lambda-1\le j\le \lambda+2\\
				0, &\text{if } \lambda+3\le j;
			\end{cases}\\
			S_{y_{n,\lambda}}^d(i-j+1) &= \begin{cases}
				0, &\text{if } i-j\le n-\lambda-3\\
				(\lambda+3)(2(i-j)-n), &\text{if } n-\lambda-2\le i-j\le \lambda+2\\
				0, &\text{if } \lambda+3\le i-j.
			\end{cases}				
		\end{split}
	\end{equation}

The proofs of the next lemmas can be consulted in \cite[Lemmas 6.3.17 and 6.3.18]{Costa2024}.

\begin{lemma}\label{lemma-y}
For $n\ge4$ and $\lambda\in\left[\lceil\frac{n-4}{2}\rceil,n-3\right] \cap \mathbb{Z}$, $y_{n,\lambda}$ in an eigenvector of $\mathcal{T}(n)$ associated with $\lambda$.
\end{lemma}

\begin{lemma}\label{lemma-xxy}
For $n\ge4$ and $\lambda\in\left[\lceil\frac{n-4}{2}\rceil,n-3\right] \cap \mathbb{Z}$, $x_{n,\lambda}, x_{n,\lambda}^+$ and $y_{n,\lambda}$ are linearly independent.
\end{lemma}

	\begin{remark}
		Note that the vectors $x_{n,\lambda}$, $x_{n,\lambda}^-$ and $x_{n,\lambda}^+$ are not linearly independent, since the sum of them is the null vector. However, if we consider just two of them, like in the previous Lemma, they are.
	\end{remark}
	
	\medskip
	
	Hence $n-3, \dots, \lceil\frac{n-4}{2}\rceil$ are eigenvalues of $\mathcal{T}(n)$. Three linearly independent eigenvectors associated with $\lambda$ were introduced, $x_{n,\lambda}, x_{n,\lambda}^+$ and $y_{n,\lambda}$, and so its multiplicity is at least $3$, except for $\frac{n-4}{2}$, when $n$ is even, whose multiplicity is at least $2$.

\begin{theorem} {\rm \cite[Th. 6.3.20]{Costa2024}} \label{IntegralGraphSpectra}
For $n \in \{1,2\}$, ${\cal T}(n) \cong K_n$. For $n=3$, $\sigma({\cal T}(n)) = \{4, 0^{[3]},-2^{[2]}\}$. Furthermore, for $n \ge 4$,
\begin{enumerate}
\item $\sigma({\cal T}(n))=\{2n-2, (n-3)^{[3]}, \dots, \frac{n-3}{2}^{[3]}, \frac{n-7}{2}^{[2]}, \frac{n-9}{2}^{[3]}, \dots, -2^{[3]}, -3^{[T(n-3)]}\}$,  when $n$ is odd;
\item $\sigma({\cal T}(n))=\{2n-2, (n-3)^{[3]}, \dots, \frac{n-2}{2}^{[3]}, \frac{n-4}{2}^{[2]}, \frac{n-8}{2}^{[3]}, \dots, -2^{[3]}, -3^{[T(n-3)]}\}$, when $n$ is even.
\end{enumerate}
\end{theorem}

\begin{proof}
For $n \in \{1, 2, 3\}$, it is immediate that ${\cal T}(n)$ is integral.\\
For $n \ge 4$, from Theorems~\ref{Th_Largest} and \ref{Th_Least} it follows that $2n-2$ and $-3$ are the largest and the least eigenvalues of ${\cal T}(n)$ with multiplicity $1$ and $T(n-3)$, respectively. Furthermore, from the results obtained in Subsections~\ref{FirstSequence} and \ref{SecondSequence}, we may conclude that $n-3, \dots, \lceil \frac{n-4}{2}\rceil, \lfloor\frac{n-7}{2}\rfloor, \dots, -2$ are also eigenvalues of ${\cal T}(n)$ whose lower bounds for their multiplicities were established as follows:
\begin{enumerate}
\item there is an eigenvalue with multiplicity $T(n-3)$ (which is $-3$);
\item there is a simple eigenvalue (which is $2n-2$);
\item there is an eigenvalue with multiplicity $2$ (which is $\frac{n-7}{2}$ and $\frac{n-4}{2}$, when $n$ is odd and even, respectively);
\item there are $n-2$ distinct eigenvalues with multiplicity $3$ (which are the remaining eigenvalues).
\end{enumerate}
Therefore, taking into account the above lower bounds, there are $T(n-3) + 1 + 2 + 3(n-2) = T(n)$ known linearly independent eigenvectors of ${\cal T}(n)$, corresponding to the $T(n)$ eigenvalues which, consequently, are totally determined.
\end{proof}

\section{Application to the $n$-Queens' graph}

In \cite{CardosoCostaDuarte2023} the authors proved that for $n \ge 4$, $-4$ is the least eigenvalue of the $n$-Queens' graph, ${\cal Q}(n)$, and its multiplicity is $(n-3)^2$.  In \cite{CardosoCostaDuarte2024}, other integer eigenvalues of ${\cal Q}(n)$ were studied and the following was proved.
\begin{enumerate}
\item $n-4$ is an eigenvalue of ${\cal Q}(3)$ with multiplicity $\frac{3+1}{2}=2$.
\item For $n \ge 4$:
      \begin{enumerate}
      \item when $n$ is even, $n-4 \in \sigma({\cal Q}(n))$, and it has multiplicity at least $\frac{n-2}{2}$;
      \item when $n$ is odd, $n-4, n-5, n-6, \dots, \frac{n-5}{2}, \frac{n-11}{2}, \dots, -2, -3 \in \sigma({\cal Q}(n))$, and  $n-4$ has multiplicity at least $\frac{n+1}{2}$.
      \end{enumerate}
\item Additionally, in \cite{CardosoCostaDuarte2024}, based in numerical computations for several values of $n$, the following conjecture was proposed.
\end{enumerate}

\begin{conjecture} {\rm \cite{CardosoCostaDuarte2024}}
For $n \ge 4$, the integers which appear above are the unique integer eigenvalues of $\mathcal{Q}(n)$. Furthermore, when $n$ is even the eigenvalue $n-4$ has multiplicity $\frac{n-2}{2}$, and when $n$ is odd the eigenvalue $n-4$ has multiplicity $\frac{n+1}{2}$. The eigenvalues $n-5, n-6, \dots, \frac{n-5}{2}, \frac{n-11}{2}, \dots, -2, -3$ are simple.
\end{conjecture}

From the computations the authors also concluded that this conjecture is true for $ n \le 100$. Note that ${\cal Q}(100)$ has $10 \, 000$ vertices.

\subsection{Graph decomposition of ${\cal Q}(n)$}

The graph decomposition of ${\cal Q}(n)$ is obtained from the vertex partition $V({\cal Q}(n)) = V_1 \cup V_2$ and an edge partition $E({\cal Q}(n)) = E_1 \cup E_2 \cup E_3$, where the $n \times n$ chessboard is divided into two right triangles, corresponding to the subgraphs $G_1 = (V_1,E_1)$ and $G_2 = (V_2,E_2)$, where the first one includes the vertices corresponding to the main diagonal of the chessboard as hypotenuse. The edges of $E_1$ and $E_2$ are chosen such that $G_1$ is a triangular graph ${\cal T}(n)$ and $G_2$ is a triangular graph ${\cal T}(n-1)$, both are integral regular graphs. Then we may consider a graph $G_{1,2} = G_1 + G_2$, that is, a graph with two connected components $G_1$ and $G_2$. The graph $G_3 = (V_3, E_3)$ is such that $V_3=V({\cal Q}(n))$ and $E_3 = E({\cal Q}(n)) \setminus E(G_{1,2})$.\\
\begin{theorem}\label{Th_Decomposition}
For $n \ge 4$, we may consider the following splitting of the adjacency matrix of the $n$-Queens' graph, ${\cal Q}(n)$,
\begin{eqnarray}
A({\cal Q}(n)) &=& A(G_{1,2}) + A(G_3)\label{MatrixSplitting1}\\
               &=& A(G_{1,2}) + \left(A(G^1_3) + \left(\underbrace{A(G^2_{3,H}) + A(G^2_{3,V})}_{A(G^2_3)}\right)\right),\label{MatrixSplitting2}
\end{eqnarray}
where each graph is obtained from the $n \times n$ chessboard as follows:
\begin{enumerate}
\item Consider the vertex partition $V({\cal Q}(n)) = V_1 \cup V_2$, where $V_1$ corresponds to the $\frac{n(n+1)}{2}$ vertices of the right triangle which includes the main diagonal as hypotenuse and $V_2$ corresponds to the remaining $\frac{(n-1)n}{2}$ vertices which defines the other right triangle. Choosing the edge subsets $E_1$ and $E_2$ such that $G_1=(V_1,E_1)$ and $G_2=(V_2,E_2)$ are the triangular graphs ${\cal T}(n)$ and ${\cal T}(n-1)$, respectively, we may conclude that the graph $G_{1,2}=G_1+G_2$ has two components $G_1 \cong {\cal T}(n)$ and $G_2 \cong {\cal T}(n-1)$. \label{GraphDecomposition1}
\item The graph $G_3 = (V_3,E_3)$ is such that $V_3=V({\cal Q}(n))$, $E_3=E({\cal Q}(n)) \setminus E(G_{1,2})$ and $A(G_3) = A(G^1_3) + A(G^2_3)$.\label{GraphDecomposition2}
      \begin{enumerate}
      \item The connected components of the graph $G^1_3$ are the cliques formed by the edges of the diagonals parallel to the second diagonal of the $n \times n$ chessboard, that is, $K_1$, $K_2$, $\dots$, $K_{n-1}$, $K_n$, $K_{n-1}$, $\dots$, $K_2$, $K_1$.\label{ConnectedComponents}
      \item Considering the vertices of $G_1$ as blue vertices and the vertices of $G_2$ as red vertices and partitioning the remaining edges of $E_3$, that is, $E(G^2_3) = E_3 \setminus E(G^1_3)$, into the subsets of horizontal and vertical edges between blue and red vertices, we obtain the graphs $G^2_{3,H}$ and $G^2_{3,V}$ such that $A(G^2_3) = A(G^2_{3,H})+ A(G^2_{3,V})$. The components of $G^2_{3,X}$, with $X \in \{H, V\}$, are the complete bipartite graphs $K_{1,n-1}$, $K_{2,n-2}$, $\dots$, $K_{n-1,1}$. Additionally, $G^2_{3,X}$ also includes $K_{n,0}$ which are just $n$ isolated vertices.
          \label{HorizontalVerticalEdges}
      \end{enumerate}
\end{enumerate}
\end{theorem}

\begin{proof}
We prove this theorem by induction on $n$, taking into account that, according to Example~\ref{ex1} below, the matrix splitting \eqref{MatrixSplitting1}-\eqref{MatrixSplitting2} associated to the graph decomposition described in items \eqref{GraphDecomposition1} and \eqref{GraphDecomposition2} is true for $n=4$.
Let us assume that the adjacency matrix of ${\cal Q}(n)$ can be split as in \eqref{MatrixSplitting1}-\eqref{MatrixSplitting2}, for some $n\ge 4$.

The $(n+1) \times (n+1)$ chessboard associated to ${\cal Q}(n+1)$ is obtained from the $n \times n$ chessboard associated to ${\cal Q}(n)$ adding a row of $n$ squares at the bottom, a column of $n$ squares at the right and a square to the right bottom position of its main diagonal. Coloring the vertices corresponding to the new bottom row with $n+1$ squares by the colour blue and colouring the vertices associated to the remaining new $n$ squares (column on the right) by the colour red, we obtain the graph $G_{1,2}=G_1+G_2$ with two components $G_1 \cong {\cal T}(n+1)$ and $G_2 \cong {\cal T}(n)$ as in item-\ref{GraphDecomposition1}.

The graph decomposition described in item \eqref{GraphDecomposition2} is easily extended from ${\cal Q}(n)$ to ${\cal Q}(n+1)$.
\end{proof}

\begin{example}\label{ex1}
Let us consider the graph ${\cal Q}(4)$ associated to the chessboard presented in Table~\ref{xadrez}

\begin{table}[h] 
\begin{center}
\begin{tabular}{|c|c|c|c|}
\hline
      {\color{blue}a} & {\color{red}b}  & {\color{red}c} & {\color{red}d} \\ \hline
      {\color{blue}e} & {\color{blue}f} & {\color{red}g} & {\color{red}h} \\ \hline
      {\color{blue}i} & {\color{blue}j} & {\color{blue}k}& {\color{red}l} \\ \hline
      {\color{blue}m} & {\color{blue}n} & {\color{blue}o}& {\color{blue}p}\\ \hline
\end{tabular}
\caption{$4 \times 4$ chessboard divided into two triangles}\label{xadrez}
\end{center}
\end{table}

Considering the vertex partition and edge partition referred in Theorem~\ref{Th_Decomposition}, it follows that ${\color{blue} G_1} \cong {\cal T}(4)$ and ${\color{red} G_2} \cong {\cal T}(3)$. From Theorem~\ref{Th_Largest} it follows that the largest eigenvalues of ${\cal T}(4)$ (and ${\cal T}(3)$) is simple and equal to $2\times 4 - 2 = 6$ (is simple and equal to $2\times 3 - 2 = 4$, respectively). Furthermore, from Theorem~\ref{Th_Least}, we know that the least eigenvalue of ${\cal T}(4)$ has multiplicity $\frac{(4-3)(4-2))}{2}=1$ and is equal to $-3$. Additionally, as it easy to conclude, $\sigma({\cal T}(3)) = \{4, 0^{[3]}, -2^{[2]}\}$. On the other hand,
$\sigma({\cal T}(4))=\{6, 1^{[3]}, 0^{[2]}, -2^{[3]}, -3 \}$. Therefore,
\begin{eqnarray*}
\sigma(G_{1,2}) &=& \{6, 1^{[3]}, 0^{[2]}, -2^{[3]}, -3 \} \underbrace{\bigcup}_{\text{with \; repetitions}} \{4, 0^{[3]}, -2^{[2]}\}\\
                &=& \{6, 4, 1^{[3]}, 0^{[5]}, -2^{[5]}, -3 \}.
\end{eqnarray*}

Analyzing the Table~\ref{xadrez} and the possible edge partitions we conclude following.

\begin{enumerate}
\item The connected components of the graph $G^1_3$ are the cliques $C_1=\{{\color{blue}a}\}$, $C_2=\{{\color{blue}e}, {\color{red}b}\}$, $C_3=\{{\color{blue}i}, {\color{blue}f}, {\color{red}c}\}$, $C_4=\{{\color{blue}m}, {\color{blue}j}, {\color{red}g}, {\color{red}d}\}$, $C_5=\{{\color{blue}n}, {\color{blue}k}, {\color{red}h}\}$, $C_6=\{{\color{blue}o}, {\color{red}l}\}$ and $C_7=\{{\color{blue}p}\}$, formed by the edges of the diagonals parallel to the second diagonal in Table~\ref{xadrez}. So, the components of $G^1_3$ are the subgraphs $K_1=G_3[C_1]$, $K_2=G_3[C_2]$, $K_3=G_3[C_3]$, $K_4=G_3[C_4]$, $K_3=G_3[C_5]$, $K_2=G_3[C_6]$ and $K_1=G_3[C_7]$, where the adjacency matrix of each clique appears as a diagonal square submatrix of $A(G^1_3)$. Therefore,
    \begin{equation}\label{sigmaG3}
    \sigma(G^1_3) = \{3, 2^{[2]}, 1^{[2]}, 0^{[2]}, -1^{[9]}\}.
    \end{equation}
\item The remaining edges can be partitioned into the horizontal (H) and vertical (V) edges (relatively to Table-\ref{xadrez}), defining two graphs, $G^2_{3,H}$ and $G^2_{3,V}$, whose componentes are the following complete bipartite subgraphs:
      \begin{enumerate}
      \item The graph $G^2_{3,H}$ is formed by $K_{1,3}$, with vertex bipartition $(\{{\color{blue}a}\},\{{\color{red}b}, {\color{red}c}, {\color{red}d}\})$, $K_{2,2}$, with vertex bipartition $(\{{\color{blue}e}, {\color{blue}f}\},\{{\color{red}g}, {\color{red}h}\})$, $K_{3,1}$, with vertex bipartition $(\{{\color{blue}i}, {\color{blue}j}, {\color{blue}k}\}, \{{\color{red}l}\})$ and $K_{4,0}$, with vertex bipartition
          $(\{{\color{blue}m}, {\color{blue}n}, {\color{blue}o}, {\color{blue}p}\}, {\color{red} \emptyset})$;
      \item The graph $G^2_{3,V}$ is formed by $K_{1,3}$, with vertex bipartition $(\{{\color{blue}p}\},\{{\color{red}d}, {\color{red}h}, {\color{red}l}\})$, $K_{2,2}$, with vertex partition $(\{{\color{blue}o}, {\color{blue}k}\}, \{{\color{red}g}, {\color{red}c}\})$, $K_{3,1}$, with vertex partition $(\{{\color{blue}n}, {\color{blue}j}, {\color{blue}f}\}, \{{\color{red}b}\})$ and $K_{4,0}$, with vertex partition $(\{{\color{blue}m}, {\color{blue}i}, {\color{blue}e}, {\color{blue}a}\}, {\color{red} \emptyset})$.
      \end{enumerate}
      Each complete bipartite graph appears as represented by its adjacency matrix in $A(G^2_{3,H})$ and $A(G^2_{3,V})$, where they are square diagonal submatrices. It is assumed that $K_{4,0}$ has no edges.\\

      Since $\sigma(A(G^2_{3,H}) = \sigma(A(G^2_{3,V})$ it follows that
      \begin{eqnarray}
      \sigma(A(G^2_{3,X}) &=& \sigma(K_{1,3}) \cup \sigma(K_{2,2}) \cup \sigma(K_{3,1}) \cup \sigma(K_{4,0}) \nonumber \\
                          &=& \{\sqrt{3}, 0^{[2]}, -\sqrt{3}\} \cup \{2, 0^{[2]}, -2\} \cup \{\sqrt{3}, 0^{[2]},-\sqrt{3}\} \cup \{0^{[4]}\} \nonumber \\
                          &=& \{2, \; \sqrt{3}^{[2]}, \; 0^{[10]}, \; -\sqrt{3}^{[2]},  \; -2 \},\label{sigma_3X}
      \end{eqnarray}
      for $X = H$ and $X = V$.
\end{enumerate}
Therefore, $A(G^2_3) = A(G^2_{3,H}) + A(G^2_{3,V})$ and  the adjacency matrix of $G_3$, is
\begin{equation}\label{G3}
A(G_3) = A(G^1_3) + \underbrace{A(G^2_{3,H}) + A(G^2_{3,V})}_{A(G^2_3)}.
\end{equation}
Taking into account \eqref{G3}, we obtain
\begin{eqnarray}
A({\cal Q}(4)) &=& A(G_{1,2}) + A(G_3) \nonumber \\
               &=& A(G_{1,2}) + \left(A(G^1_3) + \left(A(G^2_{3,H}) + A(G^2_{3,V})\right)\right).\label{AQ}
\end{eqnarray}
\end{example}

\subsection{Lower and upper bounds on the eigenvalues of ${\cal Q}(n)$}
In this section, we apply the Weyl's inequalities to the splitting \eqref{AQ}, which corresponds to the graph decomposition of ${\cal Q}(n)$, to determine lower and upper bounds for the eigenvalues of $\mathcal{Q}(n)$.

\begin{theorem}{\rm (Weyl, 1912) \cite{1912Weyl}} \label{Weil_inequalities} \\
For Hermitian matrices $A,B \in \mathbb{C}^{n\times n}$, considering the eigenvalues in non-increasing order, $\lambda_1(X) \ge \lambda_2(X) \ge \dots \ge \lambda_n(X)$ \and $1 \le i,j \le n$,
\begin{eqnarray}
\lambda_{i+j-1}(A+B)        & \le & \lambda_i(A) + \lambda_j(B), \quad i+j \le n+1, \label{Weil_inequalities1} \\
\lambda_i(A) + \lambda_j(B) & \le & \lambda_{i+j-n}(A+B).        \quad i+j \ge n+1. \label{Weil_inequalities2}
\end{eqnarray}
\end{theorem}

As direct consequence of Theorem~\ref{Th_Decomposition}, we may apply the Weyl's inequalities \eqref{Weil_inequalities1}-\eqref{Weil_inequalities2}, taking into account the matrix equation \eqref{MatrixSplitting2}.
First, let us determine the spectrum of the graphs $G_{1,2}$, $G^1_3$, $G^2_{3,H}$ and $G^2_{3,V}$.

\begin{itemize}
\item \textbf{Spectrum of \boldmath$G_{1,2}$.} Since $G_{1,2}$ has two components, one isomorphic to ${\cal T}(n)$ and other isomorphic to ${\cal T}(n-1)$, it follows that
    $$
    \sigma(G_{1,2}) = \sigma({\cal T}(n)) \underbrace{\bigcup}_{\text{with repetitions}} \sigma({\cal T}(n-1)).
    $$
    Then, from Theorem~\ref{IntegralGraphSpectra} (Theorem 6.3.20 in \cite{Costa2024}), we have the following:
    \begin{enumerate}
    \item If $n$ is odd, since
    \begin{align*}
    \sigma({\cal T}(n)) = & \left\{ 2n-2, n-3^{[3]}, \dots, \frac{n-3}{2}^{[3]}, \frac{n-7}{2}^{[2]}, \frac{n-9}{2}^{[3]}, \ldots, -2^{[3]}, -3^{[T(n-3)]} \right\}
    \end{align*}
    and
    $$
    \sigma({\cal T}(n-1)) = \left\{ 2n-4, n-4^{[3]}, \dots, \frac{n-3}{2}^{[3]}, \frac{n-5}{2}^{[2]}, \frac{n-9}{2}^{[3]}, \dots, -2^{[3]}, -3^{[T(n-4)]} \right\},
    $$
    we get
    \begin{eqnarray}
    \sigma(G_{1,2}) &=& \left\{ 2n-2, 2n-4, n-3^{[3]}, n-4^{[3]}, \dots, \frac{n-3}{2}^{[6]}, \frac{n-5}{2}^{[2]}, \frac{n-7}{2}^{[2]}, \right. \nonumber \\
                    & & \left. \frac{n-9}{2}^{[6]}, \dots, -2^{[6]}, -3^{[T(n-3)+T(n-4)]} \right\}. \label{g23odd}
    \end{eqnarray}

    \item If $n$ is even, since
    $$
    \sigma({\cal T}(n)) = \left\{ 2n-2, n-3^{[3]}, \dots, \frac{n-2}{2}^{[3]}, \frac{n-4}{2}^{[2]}, \frac{n-8}{2}^{[3]}, \dots, -2^{[3]}, -3^{[T(n-3)]} \right\}
    $$
    and
    $$
    \sigma({\cal T}(n-1)) = \left\{ 2n-4, n-4^{[3]}, \dots, \frac{n-4}{2}^{[3]}, \frac{n-8}{2}^{[2]}, \frac{n-10}{2}^{[3]}, \dots, -2^{[3]}, -3^{[T(n-4)]} \right\},
    $$
    we get
    \begin{eqnarray}
    \sigma(G_{1,2}) &=& \left\{ 2n-2, 2n-4, n-3^{[3]}, n-4^{[3]}, \dots, \frac{n-2}{2}^{[3]},\frac{n-4}{2}^{[5]}, \frac{n-8}{2}^{[5]}, \right. \nonumber \\
                    & & \left. \frac{n-10}{2}^{[3]}, \dots, -2^{[6]}, -3^{[T(n-3)+T(n-4)]} \right\}. \label{g23even}
    \end{eqnarray}
    \end{enumerate}
\item \textbf{Spectrum of \boldmath$G^1_3$.} The connected components of the graph $G^1_3$ are the cliques $K_1$, $K_2$, {\color{red}$\dots$}, $K_{n-1}$, $K_{n}$, $K_{n-1}$, $\dots$, $K_2$, $K_1$, where the adjacency matrix of each clique appears as a diagonal square submatrix of $A(G^1_3)$. Therefore,
    \begin{equation}
    \sigma(G^1_3) = \{n-1, n-2^{[2]}, \dots, 1^{[2]}, 0^{[2]}, -1^{[(n-1)^2]}\}.\label{g13}
    \end{equation}
\item \textbf{Spectrum of \boldmath$G^2_{3,H}$} which is equal to the \textbf{spectrum of \boldmath$G^2_{3,V}$}. Colouring the vertices of $G_1 \cong {\cal T}(n)$ and $G_2 \cong {\cal T}(n-1)$ as described in Theorem~\ref{Th_Decomposition}-\eqref{HorizontalVerticalEdges}, the connected components $G^2_{3,H}$ and $G^2_{3,V}$  are formed by the remaining edges of ${\cal Q}(n)$, that is, $E({\cal Q}(n)) \setminus (E(G_{2,3}) \cup E(G^1_3))$, which are edges between blue and red vertices. These remaining edges can be partitioned into the horizontal edges ($E(G^2_{3,H})$) and vertical edges ($E(G^2_{3,V})$). Then, the connected components of a graph $G^2_{3,X}$, for each $X \in \{H, V\}$, are the complete bipartite graphs $K_{i,n-i}$ for each $i \in \{1, 2, \dots, n\}$. Therefore,
    \begin{eqnarray}
    \sigma(G^2_{3,X}) &=& \underbrace{\bigcup_{i=1}^{n}{\sigma(K_{i,n-i})}}_{\text{with repetitions}} \nonumber\\
                      &=& \bigcup_{i=1}^{n-1}{\{\sqrt{(n-i)i}, 0^{[(n-1)^2+1]}, -\sqrt{(n-i)i}\}} \label{g23X}.
    \end{eqnarray}
\end{itemize}

 Taking into account that an eigenvalue of the adjacency matrix of a graph $G$ is also simple called the eigenvalue of $G$, that is, $\lambda_k(A(G)) = \lambda_k(G)$, let us introduce the following theorem.

\begin{theorem}\label{Th_LowerUpperBoundsEigenvalues}
Consider the graph decomposition described in Theorem~\ref{Th_Decomposition} and the corresponding matrix splitting \eqref{MatrixSplitting1}-\eqref{MatrixSplitting2}. Then we obtain lower and upper bounds on the eigenvalues of ${\cal Q}(n)$ as follows.\\
Assume that $1 \le i_k,j_k, r_k, s_k \le n^2$, $i_k+j_k \le n^2+1$ and $r_k+s_k \ge n^2+1$, for each $k \in \{1, 2, 3\}$.
\begin{eqnarray*}
\lambda_{i_1+j_1-1}\left({\cal Q}(n)\right) \le & \lambda_{p_1}\left(G_{1,2}\right) + \lambda_{q_1}\left(G_3\right) & \le \lambda_{r_1+s_1-n^2}\left({\cal Q}(n)\right)\\
\underbrace{\lambda_{i_2+j_2-1}}_{\lambda_{j_1}} \left(G_{3}\right) \le & \lambda_{p_2}\left(G^1_{3}\right) + \lambda_{q_2}\left(G^2_{3}\right) & \le
\underbrace{\lambda_{r_2+s_2-n^2}}_{\lambda_{s_1}} \left(G^2_{3}\right) \\
\underbrace{\lambda_{i_3+j_3-1}}_{\lambda_{j_2}} \left(G^2_{3}\right) \le & \lambda_{p_3}\left(G^2_{3,H}\right) + \lambda_{q_3}\left(G^2_{3,V}\right) & \le
\underbrace{\lambda_{r_3+s_3-n^2}}_{\lambda_{s_2}} \left(G^2_{3}\right),
\end{eqnarray*}
where $p_k=i_k$ and $q_k=j_k$ for the lower bounds and $p_k=r_k$ and $q_k=s_k$ for the upper bounds.
\end{theorem}

\begin{proof}
Let us apply Theorem~\ref{Weil_inequalities} to the matrix splitting \eqref{MatrixSplitting1}-\eqref{MatrixSplitting2}, considering starting values for the indices $i_1$, $i_2$, $i_3$ and  $j_3$ and also for the indices $r_1$, $r_2$, $r_3$ and $s_3$.
\begin{enumerate}
\item Consider the indices $1 \le i_3, j_3, r_3, s_3 \le n^2$ such that $i_3+j_3-1 = j_2$ and $r_3 + s_3 - n^2 = s_2$.
\begin{eqnarray*}
\underbrace{\lambda_{i_3+j_3-1}}_{\lambda_{j_2}} (G^2_3) & \le &\lambda_{i_3}(G^2_{3,H}) + \lambda_{j_3}(G^2_{3,V}), \text{ for } i_3+j_3 \le n^2+1 \\
\lambda_{r_3}(G^2_{3,H}) + \lambda_{s_3}(G^2_{3,V}) & \ge & 
\underbrace{\lambda_{r_3+s_3- n^2}}_{\lambda_{s_2}} (G^2_3), \text{ for } r_3+s_3 \ge n^2+1,
\end{eqnarray*}
 were $\lambda_{i_3}(G^2_{3,H}), \lambda_{j_3}(G^2_{3,V}), \lambda_{r_3}(G^2_{3,H})$ and $\lambda_{s_3}(G^2_{3,V})$ can be obtained from \eqref{g23X}.
\item Consider the indices $1 \le i_2, j_2, r_2, s_2 \le n^2$ such that $i_2+j_2-1 = j_1$ and $r_2 + s_2 - n^2 = s_1$.
\begin{eqnarray*}
\underbrace{\lambda_{i_2+j_2-1}}_{\lambda{j_1}} (G_3) & \le &\lambda_{i_2}(G^1_3) + \lambda_{j_2}(G^2_3), \text{ for } i_2+j_2 \le n^2+1 \\
\lambda_{r_2}(G^1_3) + \lambda_{s_2}(G^2_3) & \ge & 
\underbrace{\lambda_{r_2+s_2- n^2}}_{\lambda_{s_1}} (G_3), \text{ for } r_2+s_2 \ge n^2+1,
\end{eqnarray*}
were $\lambda_{i_2}(G^1_3)$ and $\lambda_{r_2}(G^1_3)$ can be obtained from \eqref{g13}.
\item Consider the indices $1 \le i_1,j_1, r_1, s_1 \le n^2$.
\begin{eqnarray*}
\lambda_{i_1+j_1-1}({\cal Q}(n))            & \le & \lambda_{i_1}(G_{1,2}) + \lambda_{j_1}(G_3), \text{ for } i_1+j_1 \le n^2+1\\
\lambda_{r_1}(G_{1,2}) + \lambda_{s_1}(G_3) & \ge & \lambda_{r_1+s_1- n^2}({\cal Q}(n)), \text{ for } r_1+s_1 \ge n^2+1,
\end{eqnarray*}
were $\lambda_{i_1}(G_{1,2})$ and $\lambda_{r_1}(G_{1,2})$ can be obtained from \eqref{g23odd} when $n$ is odd and from \eqref{g23even} when $n$ is even.
\end{enumerate}
\end{proof}

\begin{example}\label{ex2}

Considering the graph associated to the chessboard of Example~\ref{ex1}, let us apply Theorem~\ref{Th_LowerUpperBoundsEigenvalues}
starting with $i_1=3$, $i_2=1$ and $i_3=j_3=4$.

\begin{eqnarray*}
\underbrace{\lambda_{i_3+j_3-1}}_{\lambda_{j_2}=\lambda_7} \left(\underbrace{A(G^2_{3,H}) + A(G^2_{3,V}}_{A(G^2_3)})\right) & \le & \underbrace{\lambda_{4}(G^2_{3,H})}_{0}+\underbrace{\lambda_{4}(G^2_{3,V})}_{0}\\
\underbrace{\lambda_{i_2+j_2-1}}_{\lambda_{j_1}=\lambda_7}\left(\underbrace{A(G^1_3) + A(G^2_3)}_{A(G_3)}\right) & \le & \underbrace{\lambda_{1}(G^1_3)}_{3} + \lambda_{7}(G^2_3) \le 3+0 \\
\underbrace{\lambda_{i_1+j_1-1}}_{\lambda_9} \left(\underbrace{A(G_{1,2})+A(G_3)}_{{\cal Q}(4)}\right) & \le & \underbrace{\lambda_{3}(G_{1,2})}_{1} + \lambda_{7}(G_3) \le 1+3.
\end{eqnarray*}
On the other hand, starting with $r_1=15$, $r_2=16$ and $r_3=s_3=13$, we obtain
\begin{eqnarray*}
\underbrace{\lambda_{r_3+s_3-16}}_{\lambda_{s_2}=\lambda_{10}} \left(\underbrace{A(G^2_{3,H}) + A(G^2_{3,V}}_{A(G^2_3)}\right) & \ge & \underbrace{\lambda_{13}(G^2_{3,H})}_{0}+\underbrace{\lambda_{13}(G^2_{3,V})}_{0}\\
\underbrace{\lambda_{r_2+s_2-16}}_{\lambda_{s_1} = \lambda_{10}} \left(\underbrace{A(G^1_3) + A(G^2_3)}_{A(G_3)}\right) & \ge & \underbrace{\lambda_{16}(G^1_3)}_{-1} + \lambda_{10}(G^2_3) \ge -1+0 \\
\underbrace{\lambda_{r_1+s_1-16}}_{\lambda_9} \left(\underbrace{A(G_{1,2})+A(G_3)}_{{\cal Q}(4)}\right) & \ge & \underbrace{\lambda_{15}(G_{1,2})}_{-2} + \lambda_{10}(G_3) \ge -2-1.
\end{eqnarray*}


Therefore, we may conclude that $-3 \le \lambda_9({\cal Q}(4)) \le 4$. It should be observed that the eigenvalues are organized in non-increasing order, that is, $\lambda_1 \ge \lambda_2 \ge \dots \ge \lambda_{4^2}$.

\end{example}

\bigskip

\noindent \textbf{Acknowledgments.}
The authors were partially supported by CIDMA (Center for Research and Development in Mathematics and Applications) under the FCT (Portuguese Foundation for Science and Technology) Multi-Annual Financing Program for R\&D Units.

\end{document}